\documentclass[twoside, reqno, 10pt]{amsart}

\usepackage[left=4cm, right=4cm, top=3cm, bottom=2.5cm]{geometry}
\usepackage[colorlinks=true, pdfstartview=FitV, linkcolor=blue,citecolor=blue, urlcolor=blue]{hyperref}
\usepackage{amsmath}
\usepackage{paralist}
\usepackage[misc]{ifsym}
\usepackage{epsfig} 
\usepackage{epstopdf} 
\usepackage[colorlinks=true]{hyperref}
\hypersetup{urlcolor=blue, citecolor=red}
\allowdisplaybreaks

\newtheorem{theorem}{Theorem}[section]
\newtheorem{corollary}[theorem]{Corollary}

\newtheorem{proposition}[theorem]{Proposition}

\theoremstyle{definition}
\newtheorem{definition}[theorem]{Definition}
\newtheorem{remark}[theorem]{Remark}

\newcommand{\ep}{\varepsilon}

\usepackage[capitalize,nameinlink,noabbrev]{cleveref}
\usepackage{physics}

\usepackage{graphicx, graphics, psfrag, tikz, tikz-cd,svg}
\usepackage[normalem]{ulem}

\usepackage{placeins}

\usepackage{algorithm}
\usepackage{algpseudocode}

\usepackage[utf8]{inputenc}
\usepackage[T1]{fontenc}
\usepackage{mathptmx}
\usepackage{etoolbox}

\usepackage{amsfonts, amssymb, amsmath, amsthm, mathrsfs, bbm, cjhebrew, gensymb, textcomp, mathtools, dsfont, calligra}

\usepackage[numbers,sort&compress]{natbib}

\usepackage{commath, setspace, subcaption, parcolumns, multirow, accents, comment, marginnote, verbatim, empheq, enumerate, stackrel, enumitem, float, physics}

\newcommand{\vect}[1]{\boldsymbol{#1}}
\newcommand{\bu}{\vect{u}}
\newcommand{\bv}{\vect{v}}
\newcommand{\bx}{\vect{x}}
\newcommand{\bzero}{\vect{0}}

\begin{document}

\title[Inadequacy of nudging algorithms to recover non-dissipative systems]{On the inadequacy of nudging data assimilation algorithms for non-dissipative systems}
\author{Edriss S. Titi and Collin Victor}
\address[Edriss~S.~Titi]{Department of Mathematics, Texas A\&M University, College Station, TX 77843, USA and Department of Applied Mathematics and Theoretical Physics, University of Cambridge, Cambridge CB3 0WA, UK}
\address[Collin Victor]{Department of Mathematics, Texas A\&M University, College Station, TX 77843, USA}
\email[Edriss~S.~Titi]{$\quad$titi@math.tamu.edu \\ and edriss.titi@maths.cam.ac.uk}
\email[Collin Victor]{$\quad$collin.victor@tamu.edu }

\begin{abstract}
In this work, we study the applicability of the Azouani-Olson-Titi (AOT) nudging algorithm for continuous data assimilation to evolutionary dynamical systems that are not dissipative. 
Specifically, we apply the AOT algorithm to a partially dissipative variant of the Lorenz 1963 system, the Korteweg-de Vries equation (KdV) in 1D, and the 2D incompressible Euler equations. Our analysis reveals that both the Euler and KdV equations lack the finitely many determining modes property, leading to the construction of infinitely many solutions with exactly the same sparse observational data, which data assimilation methods cannot distinguish between. Simultaneously, we numerically verify that the AOT algorithm successfully recovers these counterexamples for the damped and driven KdV equation, which is dissipative. Additionally, to further support our argument, we present numerical evidence showing that the AOT algorithm is ineffective at accurately recovering solutions for a partially dissipative variant of the Lorenz 1963 system.
\end{abstract}
\maketitle

\section{Introduction}

Accurate computational prediction of nonlinear evolution systems encounters two major challenges: initialization and long-time accuracy. These refer to the difficulty in accurately initializing the system and ensuring that our computed solution remains accurate for long intervals of time, respectively. Data assimilation is a class of techniques that attempts to resolve both of these issues by incorporating observational data into the underlying evolutionary physical model. The purpose of data assimilation is to alleviate the reliance on knowing the complete state of the initial data and achieve synchronization of our simulated solution with the true dynamics by using adequately collected sparse spatio-temporal  observational data (measurements).

In this work, we are concerned with the application of data assimilation algorithms to systems that are not known to be dissipative. By dissipative, we mean that equations possess a compact absorbing ball or set in a relevant vector space (e.g. $L^2$ or $H^1$), which is a set that all solutions to the system eventually enter into and remain within for all future times. The existence of such an absorbing ball (or set) can be used to show the existence of a global attractor with finite fractal dimension, which in turn yields a finite number of parameters that determine the asymptotic in time behavior of solutions for dissipative dynamical systems.
One relevant example of data assimilation being applied to non-dissipative dynamical systems is that of the shallow water equations, which have long been used as a simplified model for the ocean dynamics.
Contrary to the case of dissipative dynamical systems \cite{Cockburn_Jones_Titi_1997}, it is unknown whether these equations have finitely many determining parameters, such as finitely many Fourier modes or nodal values which determine the solution uniquely. Yet researchers still apply data assimilation to these systems, see, e.g., \cite{Budgel_1986,Kevlahan_Khat_Protas_2019,Tirupathi_Tigran_Zhuk_McKenna_2016}.
The property, for example, of finitely many determining Fourier modes \cite{Foias_Prodi_1967, Jones_Titi_1992,Cockburn_Jones_Titi_1997} is an essential first litmus test for data assimilation to work, as it ensures the existence of a minimum length scale at which measurements and observations determine the asymptotic behavior of the system we wish to predict. Without this property, there is no guarantee that any amount of observations will be sufficient to recover the dynamics of the true solution one wishes to predict.

This work focuses on characterizing the class of systems for which one should not expect the implementation of the Azouani-Olson-Titi (AOT) continuous data assimilation algorithm, introduced in \cite{Azouani_Olson_Titi_2014,Azouani_Titi_2014}, to be effective. Since its development, this algorithm has been the subject of much study, both analytically 
\cite{Albanez_Nussenzveig_Lopes_Titi_2016,
Bessaih_Olson_Titi_2015,
Biswas_Bradshaw_Jolly_2020,
Biswas_Foias_Mondaini_Titi_2018downscaling,
Biswas_Martinez_2017,
Biswas_Price_2020_AOT3D,
Carlson_Hudson_Larios_2018,
Carlson_Larios_2021_sens,
Chen_Li_Lunasin_2021,
Celik_Olson_2023,
Diegel_Rebholz_2021,
Du_Shiue_2021,
Farhat_Jolly_Titi_2015,
Farhat_Lunasin_Titi_2016abridged,
Farhat_Lunasin_Titi_2016benard,
Farhat_Lunasin_Titi_2016_Charney,
Farhat_Lunasin_Titi_2017_Horizontal,
Foias_Mondaini_Titi_2016,
Foyash_Dzholli_Kravchenko_Titi_2014,
GarciaArchilla_Novo_Titi_2018,
GarciaArchilla_Novo_2020,
Gardner_Larios_Rebholz_Vargun_Zerfas_2020_VVDA,
Ibdah_Mondaini_Titi_2018uniform,
Jolly_Martinez_Olson_Titi_2018_blurred_SQG,
Jolly_Martinez_Titi_2017,
Jolly_Sadigov_Titi_2017,
Larios_Pei_2018_NSV_DA,
Markowich_Titi_Trabelsi_2016_Darcy,
Mondaini_Titi_2018_SIAM_NA,
Pachev_Whitehead_McQuarrie_2021concurrent,
Pei_2019,
Rebholz_Zerfas_2018_alg_nudge,
Zerfas_Rebholz_Schneier_Iliescu_2019} 
and computationally 
\cite{
Altaf_Titi_Knio_Zhao_Mc_Cabe_Hoteit_2015,
Carlson_VanRoekel_Petersen_Godinez_Larios_2021,
DiLeoni_Clark_Mazzino_Biferale_2018_inferring,
Desamsetti_Dasari_Langodan_Knio_Hoteit_Titi_2019_WRF,
Franz_Larios_Victor_2021,
Gesho_Olson_Titi_2015,
Larios_Rebholz_Zerfas_2018,
Larios_Victor_2019,
Lunasin_Titi_2015,
Larios_Pei_2017_KSE_DA_NL}. 
The AOT algorithm has been referenced a few times in the literature under the names interpolated nudging or simply nudging, among others. 
In the context of classical methods of data assimilation, nudging refers to a method developed by Anthes and Hoke \cite{Anthes_1974_JAS,Hoke_Anthes_1976_MWR} in the 1970's for numerical weather prediction. 
The AOT algorithm is superficially similar to nudging, but is technically and conceptually different. 
In the classical algorithm of Anthes and Hoke one nudges the system only at the discrete nodal points of the observational data. 
On the other hand, the AOT algorithm works by constructing a spatial approximation interpolation operator based on the observational data and then nudges at every point in the vicinity of the observation nodal point using the interpolation operator. 
Nevertheless, for simplicity, we will refer to the AOT algorithm as ``nudging'' for the remainder of this paper. 
Notably, the AOT data assimilation algorithm was designed for dissipative evolutionary nonlinear systems motivated by, and capitalizing on, the fact that the long-term dynamics of such systems are determined by the dynamics of large spatial scales (see, e.g., \cite{Cockburn_Jones_Titi_1997} and the references therein). 
Such systems include, for example, the 2D Navier-Stokes equations, the 2D Rayleigh-B\'enard convection problem, as well as the 1D Kuramoto-Sivashinsky equation, and some nonlinear reaction-diffusion systems, to name some (cf. \cite{Temam_1997_IDDS}).

The notion of dissipation for dynamical systems is often conflated with the inclusion of diffusion and viscous effects. 
Including a Laplacian term is one way to obtain a dissipative system; however, dissipation can also be realized through damping terms, as in the Korteweg-de Vries equation (see e.g. \cite{Ghidaglia_1994,Ghidaglia_1988}). 
For many dynamical systems relevant to geophysics, the model equations are assumed to be non-dissipative. 
In practice, modelers often re-introduce small artificial diffusion terms purely for numerical stability. 
We note however that the inclusion of such terms may guarantee that the model equations may be dissipative even though the underlying original system is not. 

In this work we consider a few illustrative examples of non-dissipative dynamical systems in order to demonstrate that the AOT data assimilation algorithm cannot be used to accurately recover the true solution.
We emphasize that our results are not isolated specifically to the AOT algorithm, but rather extend generically to data assimilation algorithms in general.
The failure of data assimilation algorithms to recover solutions to non-dissipative dynamical systems arises as an intrinsic property of the particular equations considered. 
By this we mean that the pathological solutions we consider here present an issue for any method of data assimilation, as the solutions are not uniquely identifiable from a fixed number of observed Fourier modes.

In this work, we consider three examples of non-dissipative dynamical systems; a partially dissipative variant of the 1963 Lorenz system, the classical Korteweg-de Vries equation in 1D, and the incompressible Euler equations in 2D. 
The 1963 Lorenz system is a chaotic system of coupled ODEs that has been used to test methods of data assimilation, including the AOT algorithm in the fully dissipative regime (see e.g. \cite{Blocher_Martinez_Olson_2018,Peng_Wu_Shiue_2023,Carlson_Hudson_Larios_Martinez_Ng_Whitehead_2022}). 
For the 1963 Lorenz system we show that removing the dissipation from even a single component is sufficient to prevent the AOT algorithm from recovering the true solution.
We consider the KdV equation in this work as, in particular, as it was proven in \cite{Ghidaglia_1988,Ghidaglia_1994} that the system has a finite-dimensional global attractor in the presence of a damping and forcing terms. 
Moreover, it was shown in \cite{Jolly_Sadigov_Titi_2017} that the nudging algorithm recovers the true solution given sufficiently many observed Fourier modes. 
Without the presence of forcing and damping, the KdV equation is a dispersive equation that has infinitely many conserved quantities, the $L^2$ norm being among them. 
The lack of dissipation and these conservation properties make this equation ideal to consider as a relatively simple 1D model to explore why data assimilation fails to recover solutions to systems without finitely many determining modes. 
While we do consider the KdV equation, the analytical findings present here extend to a broader class of PDEs that conserve the $L^2$ norm featuring the same nonlinearity.
To illustrate this, we also consider the incompressible Euler equations in 2D, which also conserve the $L^2$ norm and are not known to have finitely many determining modes.
The 2D incompressible Euler equations are of particular interest as they are the inviscid counterpart to the 2D incompressible Navier-Stokes equations, which are a classical example of a dissipative system and was the original setting of the AOT algorithm \cite{Azouani_Olson_Titi_2014}.

 The remainder of this manuscript is organized as follows: in \cref{sect:Lorenz} we consider a partially dissipative variant of the Lorenz system and show that nudging fails to recover the reference solution in this setting.
 We begin our discussion of the application of nudging to recover the solutions to the KdV equation in \cref{sect:KdV:preliminaries}, where we discuss various mathematical preliminaries regarding the formulation of the KdV equation and show that, unlike the damped and driven case, the standard version of KdV does not have the finitely many determining modes property. In \cref{sect:KdV:computational results} we  show the numerical construction of pathological examples where nudging fails to recover the true solution for the KdV equations and we additionally verify the validity of the AOT  nudging algorithm for these counterexamples when the system is dissipative.
 In \cref{sect:Euler} we extend our results from the KdV regime to the case of 2D incompressible Euler equations. In \cref{sect:Euler:preliminaries} we outline the analytical construction of pathological counterexamples for this regime and in \cref{sect:Euler:computational results} we demonstrate these findings computationally.
 We end this work in \cref{sect:conclusion} with a discussion of the conclusions we reach from our numerical results.

\section{Lorenz 1963 System}\label{sect:Lorenz}

As the focus of this work is examining the applicability of the AOT algorithm to dynamical systems that are not dissipative, we now examine a nondissipative variant of the Lorenz 1963 system. Lorenz introduced in 1963 a simplified/reduced model of atmospheric convection \cite{Lorenz_1963}, and so it is of particular interest for studying methods of data assimilation, see e.g. \cite{Hayden_Olson_Titi_2011,Blocher_Martinez_Olson_2018,Du_Shiue_2021,Hayden_2007,Pecora_Carroll_1990,Law_Shukla_Stuart_2014} and the references within. 
The 1963 Lorenz system is given by:
\begin{equation}\begin{cases}\label{eq:Lorenz}
    \dot{X} = -\sigma X + \sigma Y, & X(0) = X_0\\
    \dot{Y} = -\sigma X - Y - XZ, & Y(0) = Y_0\\
    \dot{Z} = -bZ + XY - b(r+\sigma), & Z(0) = Z_0.
\end{cases}
\end{equation}
Here $\sigma>0$ is the Prandtl number, $r>0$ is the Rayleigh number, and $b >0$ is a geometric factor. 
It is well-known that, for certain parameter choices (e.g., $\sigma  = 10$, $r = 28$, $b = 8/3$), that the Lorenz system has a chaotic global attractor \cite{Tucker_1999}.

It is also known that applying the AOT algorithm, or a synchronization scheme, (see e.g.,  \cite{Hayden_Olson_Titi_2011,Hayden_2007}), by observing only the $X$ component is sufficient to synchronize the nudged solution with the reference solution (see e.g. \cite{Hayden_Olson_Titi_2011,Blocher_Martinez_Olson_2018,Hayden_2007}). 
In particular, we can consider the following three different synchronization (nudging) schemes from \cite{Hayden_2007} (which was motivated by \cite{Olson_Titi_2008_TCFD}) to understand when data assimilation works for this system. 
\begin{equation}\label{eq:Lorenz-X}
\begin{cases}
    \dot{X} = -\sigma X + \sigma Y, & X(0) = X_0\\
    \dot{y} = -\sigma X - y - Xz, & y(0) = y_0\\
    \dot{z} = -bz + Xy - b(r+\sigma), & z(0) = z_0.
\end{cases}
    \end{equation}
\begin{equation}\label{eq:Lorenz-Y}
\begin{cases}
    \dot{x} = -\sigma x + \sigma Y, & x(0) = x_0\\
    \dot{Y} = -\sigma X - Y - XZ, & Y(0) = Y_0\\
    \dot{z} = -bz + xY - b(r+\sigma), & z(0) = z_0.
\end{cases}
    \end{equation}
    \begin{equation}\label{eq:Lorenz-Z}
\begin{cases}
    \dot{x} = -\sigma x + \sigma y, & x(0) = x_0\\
    \dot{y} = -\sigma x - y - xZ, & y(0) = y_0\\
    \dot{Z} = -bZ + XY - b(r+\sigma), & Z(0) = Z_0.
\end{cases}
    \end{equation}
In the above systems, \cref{eq:Lorenz-X,eq:Lorenz-Y,eq:Lorenz-Z}, correspond to the regimes when the observed values of $X$, $Y$, or $Z$ are inserted directly into the equations for the remaining, unobserved components. 
That is, e.g. in \cref{eq:Lorenz-X} we observe the true value of $X$ continuously in time and insert that into the evolution equations of $y$ and $z$. 
In these equations we note that numerically we should not be simulating the evolution of $X$, $Y$, or $Z$ as they appear in \cref{eq:Lorenz-X,eq:Lorenz-Y,eq:Lorenz-Z}, respectively, but rather these values are observed (given) from the true solution in \cref{eq:Lorenz} and are inserted directly into the equations for the remaining variables. 
Analysis on \cref{eq:Lorenz-X,eq:Lorenz-Y,eq:Lorenz-Z} was conducted in \cite{Hayden_2007}, where it was shown that the continuous coupling on $X$ or $Y$ alone was sufficient to obtain synchronization between the assimilated equations and the reference solution; however, the coupling seen in \cref{eq:Lorenz-Z} was not sufficient to obtain convergence in either the $x$ or $y$ components to the corresponding components $X$ or $Y$ of the true solution. We do note that in \cite{Hayden_2007} the equations for \cref{eq:Lorenz} were written in a slightly different form, however this can be easily reconciled by using the transformation $\tilde{Z} = Z + (r+\sigma)$. 

For the sake of completeness, we will discuss the convergence results for the above systems, \cref{eq:Lorenz-X,eq:Lorenz-Y,eq:Lorenz-Z} proven in \cite{Hayden_2007} and how the results break down in the absence of dissipation. Let 
\begin{equation}\label{eq:delta}
    w = (X-x,Y-y, Z-z),
\end{equation} where $(X,Y,Z)$ is the given solution of the original Lorenz system \cref{eq:Lorenz} and $(x,y,z)$ is the solution to one of the assimilated equation, \cref{eq:Lorenz-X,eq:Lorenz-Y,eq:Lorenz-Z}. Here we define $w$ with the understanding that the variable in $(x,y,z)$ missing from the assimilated equation is taken to be equal to the corresponding variable in \cref{eq:Lorenz}, e.g. when considering \cref{eq:Lorenz-X} we define $x(t):= X(t)$ for all time, $t$ and have $w = (0, Y-y, Z-z)$. The convergence results for the synchronization algorithms, \cref{eq:Lorenz-X,eq:Lorenz-Y,eq:Lorenz-Z} are given as follows:

\begin{theorem}[Hayden (2007) \cite{Hayden_2007}]\label{thm:Hayden}
    Let $(X,Y,Z)$, be the solution to \cref{eq:Lorenz}, then the following are true when $\sigma>0, b>0$, and $ r>0$.
    
    If $(x,y,z)$ is the solution to \cref{eq:Lorenz-X}, then
    \begin{equation}
        \norm{w(t)}^2_{\ell^2} \leq \norm{w_0}_{\ell^2}e^{-2t}.
    \end{equation}

    If $(x,y,z)$ is the solution to \cref{eq:Lorenz-Y}, then
    \begin{equation}\label{eq:J-bound}
    \begin{split}
        \norm{w(t)}^2_{\ell^2} &\leq  |X(0) - x(0)| e^{-2\sigma} + |Z(0) - z(0)| e^{-bt}\\
        &\quad + \frac{J|X(0)-x(0)|^2}{b(2\sigma - b)}\left(e^{-b t} - e^{-2\sigma t}\right),
    \end{split}
    \end{equation}
    where
    \begin{equation} 
        J = \begin{cases}
            r^2, & b \leq 2\\
            \frac{b^2 r^2}{4(b - 1)}, &b \geq 2
        \end{cases}
    \end{equation}
\end{theorem}

The proofs of \cref{thm:Hayden} are derived by taking the difference of the assimilated equation (\cref{eq:Lorenz-X} or \cref{eq:Lorenz-Y}) with \cref{eq:Lorenz}, summing the squares of the derivatives, and applying Gr\"onwall's inequality to the resulting quantity. The $J$ estimate in \cref{eq:J-bound} arises from the use of the bound
\begin{equation}\label{eq:Y-bound}
    \limsup_{t\to\infty} \left(Y(t)\right)^2 \leq J,
\end{equation}
which was derived in \cite{Doering_Gibbon_1995_book} and was used to bound the contribution from the nonlinear term in the $z$ equation. 
It is worth noting that the bound \cref{eq:Y-bound} is valid uniformly and globally in time for solutions in the global attractor that we intend to recover using data assimilation algorithms.

The analysis done in the proof of \cref{thm:Hayden} is dependent on the particular parameters and requires $\sigma>0$, $b >0$, and $r>0$. To illustrate how these sorts of estimates break down  in the absence of dissipation, we consider the case when both $X$ and $Y$ are observed, i.e., are given from the true solution of \cref{eq:Lorenz}.  This leads to the following system of equations
\begin{equation}\label{eq:Lorenz-XY}
    \begin{cases}
    \dot{X} = -\sigma X + \sigma Y, & X(0) = X_0\\
    \dot{Y} = -\sigma X - Y - XZ, & Y(0) = Y_0\\
    \dot{z} = -bz + XY - b(r+\sigma), & z(0) = z_0.
\end{cases}
\end{equation}
Now, if we take $w = z - Z$, where $Z$ solves \cref{eq:Lorenz}, we see from \cref{eq:Lorenz,eq:Lorenz-XY} that $w$ satisfies the following equation
\begin{equation}\label{eq:Lorenz-w}
   \dot{w} = -bw, \quad w(0) = z_0 - Z_0.
\end{equation}
From here, we can see that if $b>0$, we can multiply both sides by $w$ and apply Gr\"onwall's inequality to obtain
\begin{equation}
    |w(t)|^2 \leq |w(0)|^2e^{-bt},
\end{equation}
which yields convergence of $z$ to $Z$ exponentially fast in time. If we instead consider $b=0$, removing the dissipation on the $Z$ component of \cref{eq:Lorenz}, then we now have that $\dot{w} = 0$, and so
    $w(t) = z_0 - Z_0$ holds for all time.
Therefore, we have that the absence of dissipation on the $Z$ component prevents synchronization between the assimilated solution and the reference solution, even when both the $X$ and $Y$ components are observed continuously in time.

The above discussion has been on the synchronization data assimilation scheme, given by the direct insertion, as it was first introduced in \cite{Olson_Titi_2008_TCFD}, of the observed quantities into the model equations. We now consider the nudged Lorenz system for the observed $X$ and $Y$ components of \cref{eq:Lorenz}, given as follows:
\begin{equation}\label{eq:Lorenz:nudge}
    \begin{cases}
         \dot{x} = -\sigma x + \sigma y + \mu(X - x)\\
    \dot{y} = -\sigma x - y - xz + \mu(Y-y)\\
    \dot{z} = -bz + xy - b(r+\sigma)
    \end{cases}
\end{equation}
Here $\mu > 0$ is the nudging coefficient, which controls the strength of the feedback-control terms. For simplicity, we utilize the same nudging parameter, $\mu$, in the equations for $x$ and $y$, however this is not strictly necessary. 

At any fixed time, $t$, the solutions, $(x(t;\mu), y(t;\mu), z(t;\mu))$, to the nudged system, \cref{eq:Lorenz:nudge}, are expected to converge to the solution $(X(t),Y(t), z(t))$ of the synchronization algorithm, given by \cref{eq:Lorenz-XY} as $\mu \to \infty$. This can be justified rigorously by applying the argument in \cite{Carlson_Farhat_Martinez_Victor_2024_ISYNC} to the Lorenz 1963 system, provided both methods are initialized with the same initial data. An analogous statement was made in \cite{Peng_Wu_Shiue_2023} for the Lorenz system in particular (with $b>0$), where it was shown that the solution, $x(t;\mu)$, of \cref{eq:Lorenz:nudge}, converges to $X(t)$ at any fixed time, $t$, as we send $\mu \to \infty$ when the nudging term is present only on the equation for $x$ in \cref{eq:Lorenz:nudge}.  In the case where $b = 0$, if we assume \textit{a priori} that the solutions  of \cref{eq:Lorenz:nudge} converge to those of \cref{eq:Lorenz-XY} as $\mu \to \infty$, one can deduce using the reverse triangle inequality, that $\norm{w(t)}^2_{\ell^2} \geq |z_0 - Z_0|^2$, where $w$ is given as defined previously in \cref{eq:delta}. Therefore, we expect the nudged Lorenz system, \cref{eq:Lorenz:nudge}, to fail to recover the true solution in the same way that \cref{eq:Lorenz-XY} fails when dissipation is removed from the equations for $z$.

We now demonstrate numerically that this dissipation mechanism (when $\sigma, b > 0$) of the Lorenz system is necessary for data assimilation methods to work. In particular, we will be considering a partially dissipative variant of the Lorenz system by taking $b = 0$, i.e., eliminating the dissipation mechanism only on the $Z$ component. The analytical results regarding the nudging algorithm applied to the Lorenz equations have been done when $b>0$ and are qualitatively similar to those we reference above in \cref{thm:Hayden} from \cite{Hayden_2007}, see e.g. \cite{Blocher_Martinez_Olson_2018, Peng_Wu_Shiue_2023,Carlson_Hudson_Larios_Martinez_Ng_Whitehead_2022}. That is, the solutions of \cref{eq:Lorenz:nudge} synchronize with those of \cref{eq:Lorenz} exponentially fast in time for large enough nudging parameter values and with the exponential rate depending on the particular parameters of \cref{eq:Lorenz}. While nudging is only required on either the $X$ or $Y$ components in order to drive convergence, we will utilize nudging on both the $X$ and $Y$ components to emphasize that, even with additional nudging, the method is not effective when the dissipation on the $z$-component is eliminated (i.e. when $b = 0$).

The Lorenz system is implemented numerically as in \cite{Blocher_Martinez_Olson_2018, Peng_Wu_Shiue_2023} using an explicit Euler scheme for the time discretization.
All simulations depicted in this study were performed Matlab (R2023a) using our own code.
As in \cite{Peng_Wu_Shiue_2023}, we will utilize $\Delta t = 10^{-4}$ for our timestep.  
In many works the reference solution is initialized by running the solution forward in time until it is approximately on the global attractor (see e.g. \cite{Blocher_Martinez_Olson_2018,Hayden_Olson_Titi_2011}), however as we are attempting to test nudging for non-dissipative or partially dissipative systems, which need not possess global attractors, we will instead be using the initial data $(X,Y,Z) = (30,40,50)$ as in \cite{Peng_Wu_Shiue_2023}.

First, we verify the validity of nudging on the Lorenz system with full dissipation. Using the initial data and parameter choices discussed previously we apply nudging to the Lorenz system with $b = 8/3$. The results can be seen in \Cref{fig:lorenz:dissipative}, where we have applied nudging to only the first two components of \cref{eq:Lorenz:nudge}. We see that all variables converge to an error level of approximately $10^{-12}$ with dips below this error level in each component. These dips are expected due to the oscillations of the reference solution. Note that here we see exponential convergence to the true solution, as we expect in this case.

\begin{figure}
    \centering
    \includegraphics[width=0.95\linewidth]{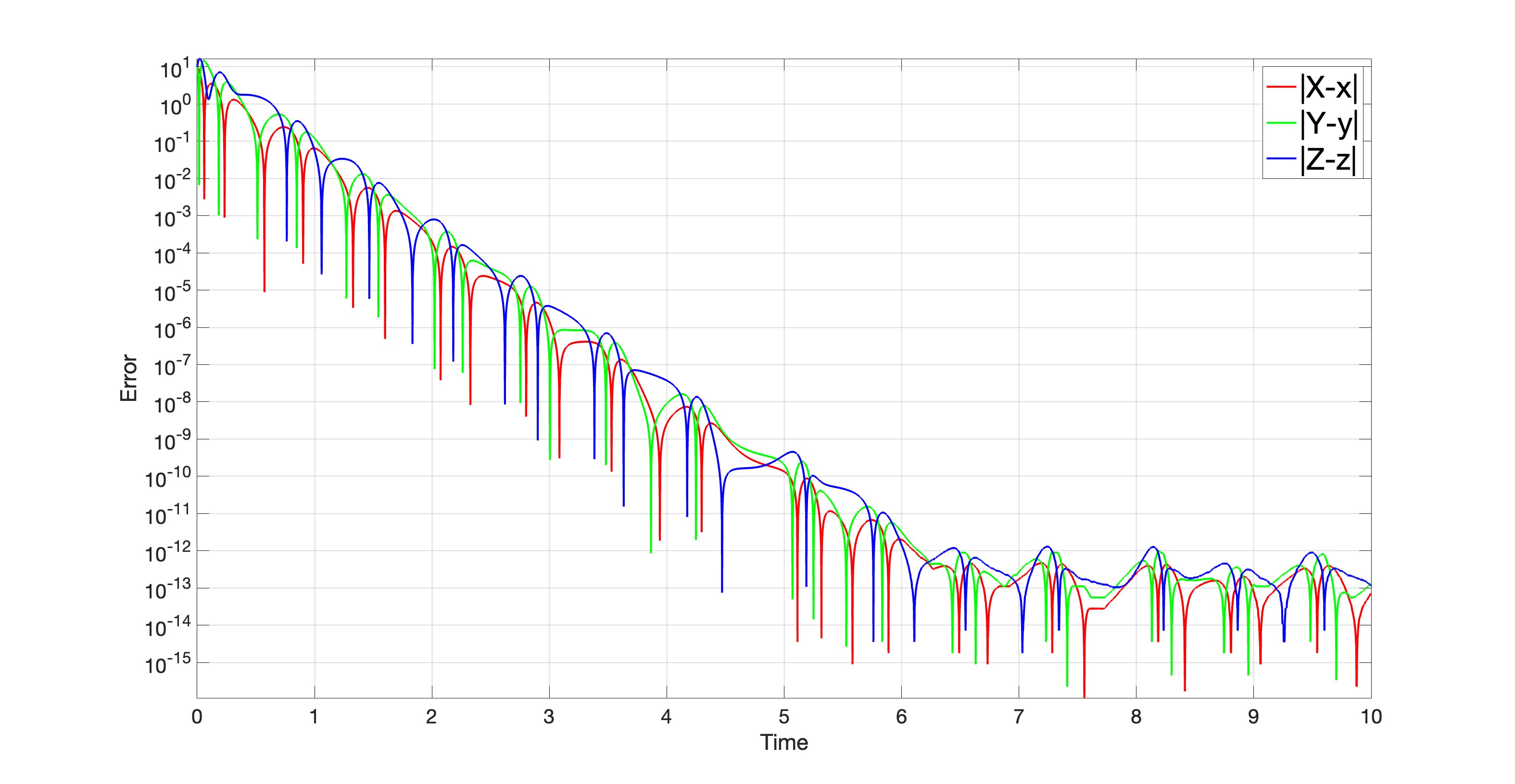}
    \caption{Absolute error for each component over time for the application of nudging to the Lorenz system with $\mu = 10$. Initial data is $(X,Y,Z) = (30, 40, 50)$ and $(x,y,z) = (20,30,40)$. Here $b = 8/3$.}
    \label{fig:lorenz:dissipative}
\end{figure}

Next, we examine what happens when we apply nudging for the partially dissipative Lorenz system, i.e., with $b = 0$. The results can be seen in \Cref{fig:lorenz:nondissipative}. Note that here we still see exponential decay to the true solution in both $x$ and $y$. While nudging appears to affect $z$ initially, this effect quickly fades after a short interval of time and the error appears to stabilize to be approximately constant $10^1$. The error level that is attained appears to be directly linked to the difference in the initial conditions for the $z$ component. To illustrate this, we plotted the  $\ell^2$ (Euclidean norm) error over time when both the reference and nudged equations are initialized with the same initial data, except for the $z$ component. Here  $Z(t_0) = 50$ and $z(t_0) = 50 + \delta$, where $\delta$ ranges in values from $10^{-1}$ to $10^{-12}$. 
The results of these tests can be seen in \cref{fig:lorenz:sensitivity} and they illustrate that the $\ell^2$ error is highly sensitive to small perturbations in the initial condition  for $z$. The behavior seen in \cref{fig:lorenz:sensitivity} appears to align with what one should expect from our earlier discussion of \cref{eq:Lorenz-w} for the synchronization algorithm. We obtain an initial exponential decay in the error in the $z$ component which is driven by the convergence in both $x$ and $y$. This initial convergence levels off and becomes approximately constant over time.

\begin{figure}
    \centering
    \includegraphics[width=0.95\linewidth]{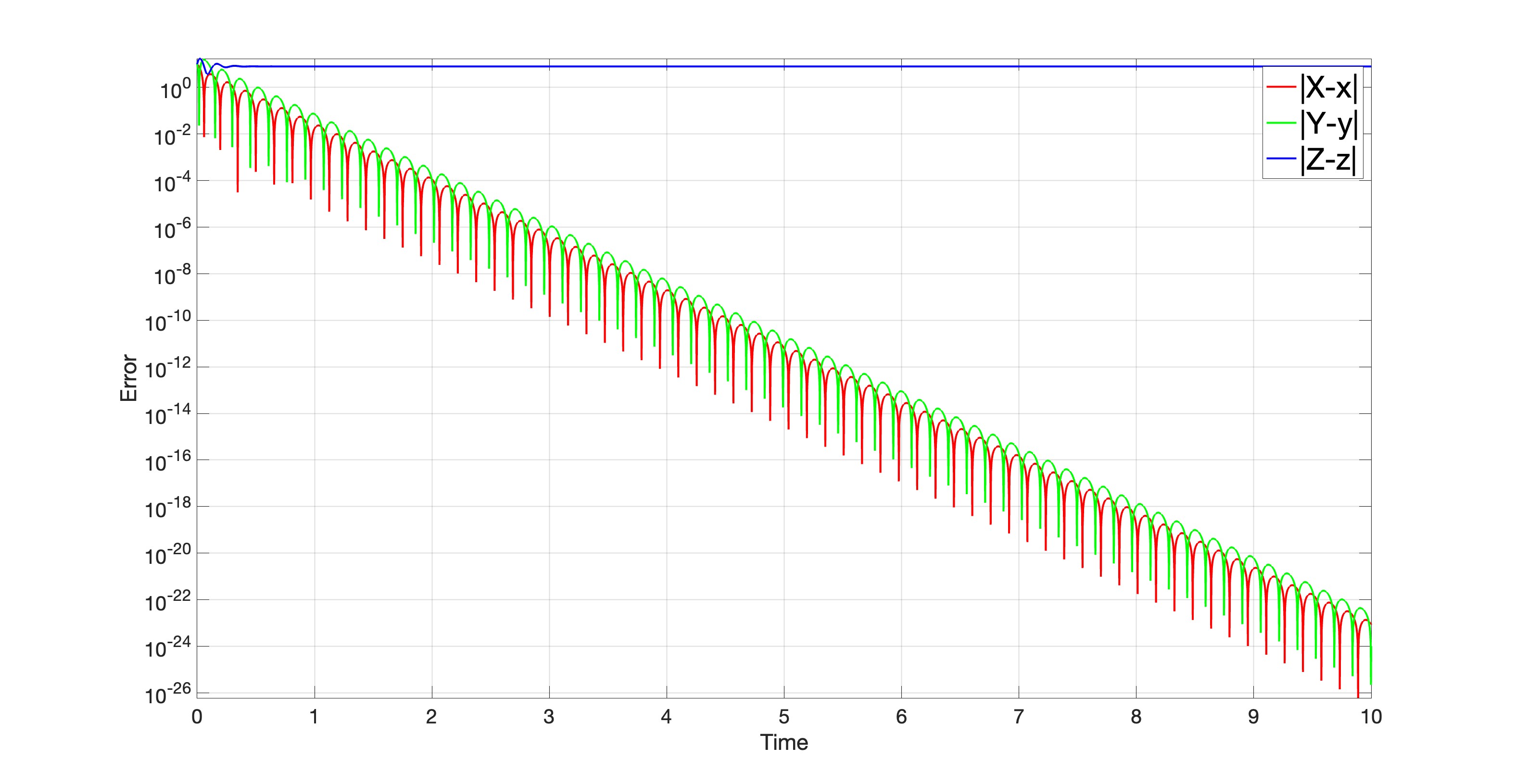}
    \caption{Absolute error for each component over time for the application of nudging to the Lorenz system with $\mu = 10$. Initial data is $(X,Y,Z) = (30, 40, 50)$ and $(x,y,z) = (20,30,40)$. Here $b = 0$.}
    \label{fig:lorenz:nondissipative}
\end{figure}

\begin{figure}
    \centering
    \includegraphics[width=0.95\linewidth]{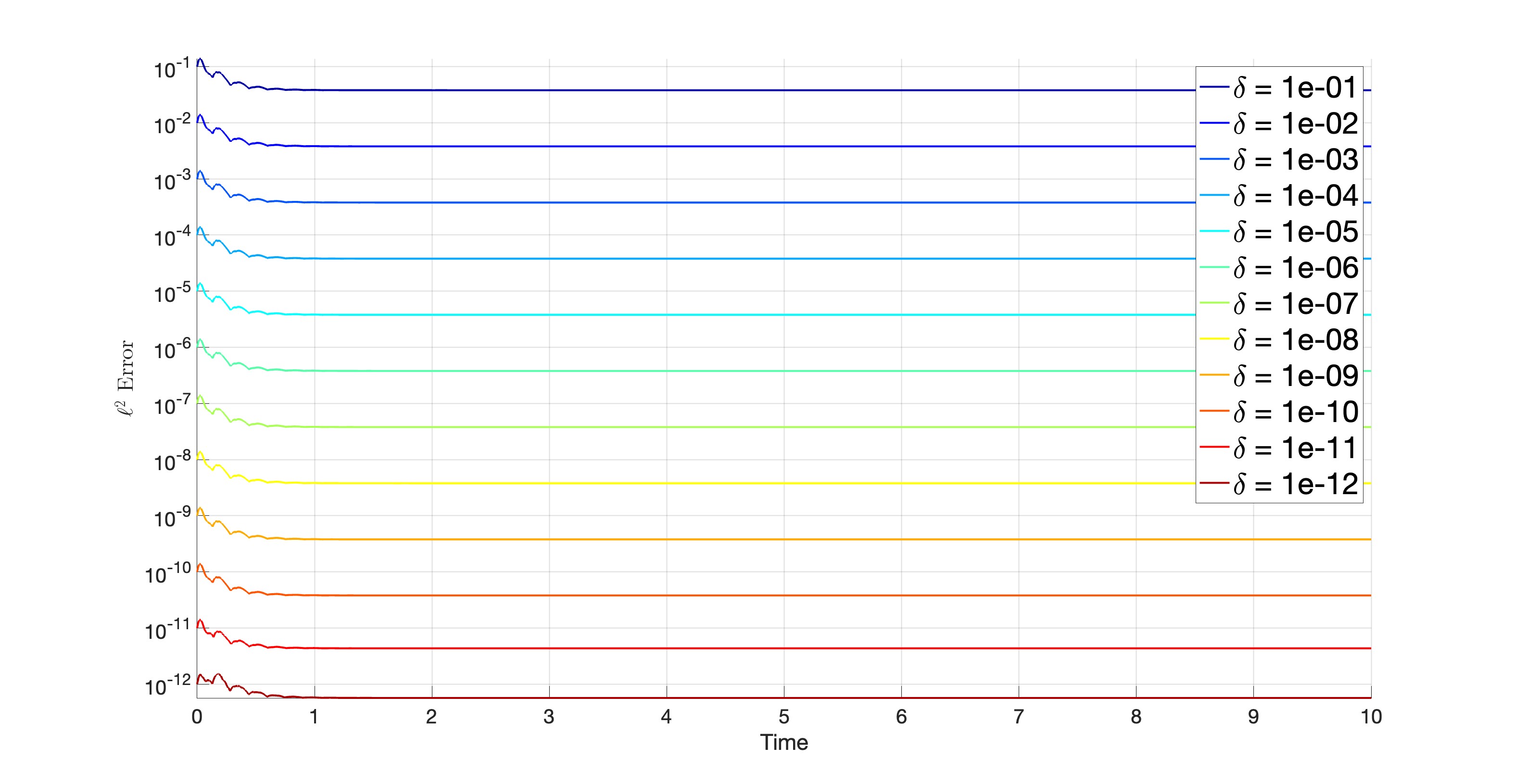}
    \caption{Effect of perturbations of the true initial condition on convergence for nudging for Lorenz. Here $Z(t_0) = 50$ and $z(t_0) = 50 + \delta$, with $X(t_0) = x(t_0) = 20$ and $Y(t_0) = y(t_0) = 30$ with $\mu = 10$. Parameters used were $\sigma = 10$, $\rho = 28$, and $b = 0$.}
    \label{fig:lorenz:sensitivity}
\end{figure}

\section{1D PDEs - Korteweg-de Vries}

\subsection{Preliminaries and Analytical Results}\label{sect:KdV:preliminaries}

Generally speaking, the nudging algorithm is described as follows. Consider the following evolution system:
\begin{equation}\label{eq:generic system}
u_t = F(u),
\end{equation}
where here $F$ is a given differential operator that is potentially nonlocal and nonlinear. We will specifically be considering solutions of \cref{eq:generic system} that lie in, e.g., the phase space $L^2$ subject to periodic boundary conditions. 

In the remainder of this section we will utilize $H^m_{\text{per}}$, with $m$ being a non-negative integer, to denote the closure of the set of $L$-periodic trigonometric polynomials on $\mathbb{R}$, with respect to the standard Hilbert norm:
\begin{equation}
    \norm{f}_{H^{m}_{\text{per}}} := \left(\sum_{k = 0}^m \int_0^L \abs{\frac{\partial^k f}{\partial x^k} }^2dx\right)^{\frac{1}{2}}. 
\end{equation}
Here $H^0_{\text{per}} = L^2_{\text{per}}$, with $L^2_{\text{per}}$ denoting the space of $L^2$ functions subject to $L$-periodic boundary conditions.
For simplicity of notation we will drop the subscript ``per'' as we will only be considering periodic functions in this work.

We assume that \cref{eq:generic system} is globally well-posed  and for such systems we define the related nudged system, introduced in \cite{Azouani_Olson_Titi_2014,Azouani_Titi_2014}, as follows:
\begin{equation}
v_t = F(v) + \mu I_h(u-v).
\end{equation}
Here $I_h$ is a linear spatial interpolant with associated length scale $h$ and $\mu$ is a constant feedback-control nudging parameter. For our purposes we will consider only $I_h := P_M$, where $P_M$ is given to be a projection operator onto the lowest $M$ Fourier modes, i.e. $h\sim \frac{L}{M}$, where $L$ is the period for the underlying evolutionary equation \cref{eq:generic system}. As we require $P_M u$ to be well defined we will only consider solutions to \cref{eq:generic system} that are in $L^2$ for all positive time.

An important property that is often required for the success of data assimilation algorithms is that the underlying system has finitely many determining modes. 
This property was first introduced in \cite{Foias_Prodi_1967} for the 2D incompressible Navier-Stokes equations and later extended to other dissipative systems, see e.g. \cite{Jones_Titi_1992, Cockburn_Jones_Titi_1997,Foias_Manley_Rosa_Temam_2001}. 
We now give the definition of this property.
\begin{definition}[Determining modes, as given in \cite{Foias_Prodi_1967, Jones_Titi_1992, Cockburn_Jones_Titi_1997,Foias_Manley_Rosa_Temam_2001}]
    We say that \cref{eq:generic system} has {\bfseries finitely many determining modes} if there exists a positive integer $M>0$ such that for any two given solutions $u_1$ and $u_2$ of \cref{eq:generic system} the following holds:\\
    If 
    \begin{equation}
        \lim_{t\to \infty}\norm{P_M(u_1(t) - u_2(t))}_{L^2} = 0,
    \end{equation}
    then 
    \begin{equation}
        \lim_{t\to \infty} \norm{u_1(t) - u_2(t)}_{L^2} = 0.
    \end{equation}
    Here $P_M$ is given to be the orthogonal projection onto the lowest $M$ Fourier modes.
\end{definition}

For the purpose of this study we will let $F$ result in the KdV equation:
\begin{equation}\label{eq:KdV}
    u_t + uu_x + \delta^2 u_{xxx} = 0.
\end{equation}
Here $u: \mathbb{R} \times [0,\infty) \to \mathbb{R}$ represents the amplitude of a wave and $\delta$ is the dispersive coefficient. We consider the KdV equation equipped with the periodic boundary condition $u(x,t) = u(x+2,t)$, for all $x\in \mathbb{R}$ and all time $t\geq 0$ with spatial mean zero. Hence, the period in this case is $L = 2$. Choosing the length of the spatial domain to be 2 is arbitrary and was chosen as this choice of domain appears in the literature, see, e.g., \cite{Zabusky_Kruskal_1965}.

The corresponding nudged system is given as follows:
\begin{equation}\label{eq:KdV nudged}
    v_t + vv_x + \delta^2 v_{xxx} = \mu P_M(u - v).
\end{equation}
Here $\mu > 0$ is a feedback-control nudging parameter and $P_M$ is given to be the projection operator onto the lowest $M$ Fourier modes. Here $u$ is the unknown reference solution of \cref{eq:KdV}, for which the first $M$ Fourier modes are observed (given) and we aim to reproduce it.

Below and through the remainder of the paper, we will use the notation $\hat{f}$ to reference the Fourier modes of an arbitrary function $f$ in $L^2$. That is, 
\begin{equation}
    \hat{f}(k,t):= \int_{0}^2 f(x,t) e^{-\pi ik\cdot x}dx.
\end{equation}
For simplicity of presentation, we omit the time index and, unless otherwise stated, we write $\hat{f}_k := \hat{f}(k,t)$.

We will now state and prove some propositions that will be useful in our numerical results.
\begin{proposition}\label[proposition]{prop:1}
    Let $k\in \mathbb{Z}\setminus \{0\}$ and $\phi \in L^2$ with $\phi$ having spatial mean zero. The following hold true:
    \begin{enumerate}[label=(\roman*)]
        \item $\phi(x + \frac{L}{k}) = \phi(x)$ for a.e. $x\in \mathbb{R}$ if and only if
        \begin{equation}
            \phi(x) = \sum_{n\in k\mathbb{Z}\setminus\{0\}} \hat{\phi}_n e^{i \frac{2\pi}{L} n x}.
        \end{equation}

        \item Let $M$ be a positive integer and let $k$ be an integer satisfying $\abs{k}\geq M$. Suppose $\phi(x + \frac{L}{k}) = \phi(x)$ for a.e. $x\in \mathbb{R}$, then
        \begin{equation} P_M \phi = 0 \end{equation}
    \end{enumerate}
\end{proposition}

\begin{proof}
The proof of part (i) is trivially true and part (ii) follows directly from part (i) using the fact that $\hat{\phi}_0 = 0$, as $\phi$ has zero spatial mean.
\end{proof}

\begin{proposition}\label[proposition]{prop:2}
    Suppose $k\in \mathbb{Z}\setminus \{0\}$ and $u^{in}, f \in H^2$ such that $f(x+\frac{L}{k}) = f(x)$ and $u^{in}(x+\frac{L}{k}) = u^{in}(x)$ for all $x\in \mathbb{R}$ and $u^{in}$ and $f$ both have spatial mean zero. Let $u(x,t)$ be the corresponding solution of the damped and driven KdV equation:
    \begin{equation}\label{eq:KdV damped}
        u_t + uu_x + \delta^2 u_{xxx} + \gamma u = \beta f,\\
        u(x,0) = u^{in}(x),
    \end{equation}
    with $\gamma \geq 0$ and $\abs{\beta} \geq 0$, subject to periodic boundary conditions with period $L$.

    Then $u(x+\frac{L}{k},t) = u(x,t)$ for a.e. $x\in \mathbb{R}$ and all $t\in \mathbb{R}$.
    Moreover, for every $0<T<\infty$, the solution $u$, restricted to the time interval $[-T,T]$, is continuous and bounded, as in 
    \cite{Bona_Smith_1975}.
\end{proposition}

\begin{proof}
    Let $v(x,t) = u(x+\frac{L}{k},t)$ and notice that 
    $$v(x,0) = v^{in}(x) = u^{in}(x+\frac{L}{k},0) = u^{in}(x).$$
    Moreover, because of the $\frac{L}{k}$ periodicity of $f(x)$ one has that $v(x,t)$ is also a solution of \cref{eq:KdV damped}. As $v(x,0) = u^{in}(x)$, then by the uniqueness of solutions to \cref{eq:KdV damped} we obtain that $v(x,t) = u(x,t)$, hence we have that $u(x+\frac{L}{k},t) = u(x,t)$. The fact that $u$ restricted to $[-T,T]$ is continuous and bounded follows immediately from Theorem 10 in \cite{Bona_Smith_1975}.
\end{proof}

\begin{corollary}\label[corollary]{cor}
    Under the conditions of \Cref{prop:2} the space of periodic functions with period $\frac{L}{k}$ is invariant under the solution of \cref{eq:KdV damped}.
\end{corollary}

\begin{proposition}\label[proposition]{prop:4}
    Suppose $k\in \mathbb{Z}\setminus \{0\}$ is given and suppose $u^{in}_j \in H^2$, for $j = 1,2$, satisfying
    \begin{equation}
        u^{in}_j(x+\frac{L}{k}) = u_j^{in}(x)
    \end{equation}
    for a.e. $x\in \mathbb{R}$ and $u_j^{in}$ has spatial mean zero. Moreover, assume $\norm{u^{in}_1}_{L^2} \neq \norm{u^{in}_2}_{L^2}$ and let $u_j(x,t)$ be the corresponding solutions of \cref{eq:KdV}. Then
    \begin{enumerate}[label=(\roman*)]
        \item $u_j(x+\frac{L}{k},t) = u_j(x,t)$ for a.e. $x\in \mathbb{R}$ and for $j = 1,2$ and all $t\in \mathbb{R}$.
        \item $P_M u_1(\cdot,t) = P_M u_2(\cdot,t)$ for any positive integer $M < \abs{k}$, and every $t\in \mathbb{R}$.
        \item $\limsup\limits_{t\to \infty} \norm{u_1(\cdot, t) - u_2(\cdot, t)}_{L^2} > 0.$
    \end{enumerate}
\end{proposition}

\begin{proof}
    As \cref{eq:KdV damped} is a particular case of \cref{eq:KdV damped} (i.e. $\gamma = 0, \beta = 0$), then (i) follows immediately from \Cref{prop:2}. Part (ii) then follows immediately from part (i) combined with \Cref{prop:1}.

Let us now prove part (iii). Let $u_1$ and $u_2$ be two solutions to \cref{eq:KdV} under the assumptions of the proposition. Without loss of generality, suppose $\norm{u_1^{\text{in}}}_{L^2} > \norm{u_2^{\text{in}}}_{L^2}$. We recall that the $L^2$ norm is a conserved quantity for \cref{eq:KdV} and thus for $j = 1,2 $ we have that $\norm{u_j(t)}_{L^2} = \norm{u_j^{\text{in}}}_{L^2}$ holds for all $t\in \mathbb{R}$. The result follows from an application of the reverse triangle inequality as
\begin{equation}
    0< \norm{ u_1(\cdot, t)}_{L^2} - \norm{u_2(\cdot,t)}_{L^2} \leq \norm{ u_1(\cdot,t) - u_2(\cdot,t)}_{L^2}
\end{equation}
holds for all time $t\in \mathbb{R}$.

\end{proof}

\begin{remark}
\Cref{prop:4} can be adjusted without change to allow $u_1$ and $u_2$ to be initialized with periods $\frac{L}{k_1}$ and $\frac{L}{k_2}$, respectively. The only changes in this case are that $u_j$ remain $\frac{L}{k_j}$ periodic for all $t\in \mathbb{R}$ in (i) and $M < \min\{\abs{k_1},\abs{k_2}\}$ is required for (ii).    
\end{remark}

Now, using these propositions, we show that \cref{eq:KdV} does not have the finitely many determining modes property.

\begin{theorem}\label{thm:determining modes}
The Korteweg-de Vries equation \cref{eq:KdV} subject to periodic boundary conditions with period $L$ does not possess the finitely many determining modes property in $L^2$.    
\end{theorem}

\begin{proof}
    The proof of this follows straightforwardly from \Cref{prop:4}. Assume, by way of contradiction, that the KdV equation enjoys the finitely many determining Fourier modes property in $L^2$. That is, there exists some $M \in \mathbb{N}$ such that for any two solutions $u_1$ and $u_2$ to \cref{eq:KdV} that if
    \begin{equation}
        \lim_{t\to \infty}\norm{P_M(u_1(t) - u_2(t))}_{L^2} = 0,
    \end{equation}
    then 
    \begin{equation}
        \lim_{t\to \infty} \norm{u_1(t) - u_2(t)}_{L^2} = 0.
    \end{equation}

We simply use $k = M + 1$, and consider $u_1$ and $u_2$ to be any two solutions to \Cref{eq:KdV} such that $u_1$ and $u_2$ both are smooth enough and have zero spatial mean, are $\frac{L}{k}$ periodic initially, and $\norm{u_1^{in}}_{L^2} \neq \norm{u_2^{in}}_{L^2}$. Two such solutions must exist  as, e.g., one could take $u_1$ to be the zero solution and $u_2$ to be the solution generated from the initial data $u_2^{in} = \cos(\frac{2\pi}{L} k x).$

Now, we immediately obtain a contradiction using \Cref{prop:4}.
\end{proof}

\begin{remark}
    The proof of the above can be adapted almost without change to other non-dissipative equations with a similar nonlinearity and conservation of the $L^2$ norm or a similar quantity.
    We will show in \cref{sect:Euler} that a similar statement works for the 2D incompressible Euler equations.
  
\end{remark}

\subsection{Computational Results}\label{sect:KdV:computational results}

In this section we will detail the numerical methods we utilize to simulate \cref{eq:KdV,eq:KdV nudged}. 
For this computational study we implemented the KdV equation on $\mathbb{R}$, subject to periodic B.C. with period 2 using pseudo-spectral methods. 
Specifically, we implement the equations using an integrating factor method to compute the linear term exactly with a fourth-order Runge-Kutta scheme for the timestepping with the nonlinear term computed explicitly. 
For an overview of integrating factor schemes see e.g. \cite{Kassam_Trefethen_2005, Trefethen_2000_spML} and the references contained within. 
Our particular choice of numerical scheme, which is given \cref{alg:IFRK4}, was chosen as it allows for larger timesteps than traditional schemes for KdV, such as the Zabusky-Kruskal explicit leapfrog scheme seen in the literature \cite{Zabusky_Kruskal_1965}.

In the simulations of  \cref{eq:KdV}, depicted below, we used different spatial resolutions depending on the particular choice of initial profile and dispersion coefficient, $\delta^2$. We ensured in our simulations that $\Delta t \lesssim {\Delta x}^{2} $, which was consistent with the implementation  of this timestepping scheme  in chapter 10 of  \cite{Trefethen_2000_spML} and appeared to be numerically stable for all spatial resolutions we used  in this work.  Here $\Delta t = t_{n+1} - t_n$ and $\Delta x = \frac{2}{N}$ represent the temporal and spatial resolution scales with $N$ indicating the spatial resolution, given in powers of $2$. While we do not explicitly state the spatial resolution for each simulation, one can see in the spectral plots, i.e. 
in \cref{fig:zero separation,fig:zero control,fig:solenoid control,fig:large time failure,fig:large time success,fig:damped and driven},
that the x-axis covers $[1, N/2]$. $N$ ranges between $2^7$ and $2^{11}$, depending on the amount of Fourier modes required to accurately simulate a given solution. 
$N$ is chosen specifically to ensure that the solution is well resolved, which here means that Fourier coefficients of the reference solution decay to machine precision (given by $\ep = 2.216\times 10^{-16}$) before the dealiasing cutoff, i.e. the vertical red line featured in 
\cref{fig:zero separation,fig:zero control,fig:solenoid control,fig:large time failure,fig:large time success,fig:damped and driven}.
We utilize 2/3 truncation  for dealiasing to accurately calculate the nonlinear term, i.e., we set the last 1/3 of the Fourier modes to zero to prevent aliasing effects. 
Additionally, while the red vertical line on spectral plots (see 
\cref{fig:damped and driven:b,fig:damped and driven:d,fig:large time failure:b,fig:large time success:b} 
denotes the dealiasing cutoff, the vertical blue line denotes the observational window cutoff, that is, the Fourier modes to the left of and including the blue vertical line are observed and used in the computation of the nudging term.

\begin{algorithm}[H]
    \caption{Fourth Order Runge-Kutta with Integrating Factor Scheme  with Nudging Term}
    \label{alg:IFRK4}
\begin{algorithmic}[1] 
\Require
Observational operator $H:= P_M$ (projection onto the lowest $M$ Fourier modes), \\
Nudging parameter $\mu$,\\
Nonlinear function $N(u,t)$,\\
Linear function $L := -\delta^2 i k^3 + \gamma$ for integrating factor,\\
${u}_n$ and ${v}_n$, the value of solutions $u$ and $v$ at previous iteration,\\
$\hat{u}_n = \mathcal{F}(u_n)$ and $\hat{v}_n = \mathcal{F}(v_n)$ represent the Fourier transforms of $u_n$ and $v_n$, respectively. Note that $\hat{v}_{n}(k)$ is a function of the wave index, $k$, and the indices present here indicate the current timestep.
\bigskip

\State Forward time integration $\hat{u}_{n+1} := M(\hat{u}_n).$
\begin{align*}
k_1 &= N(\hat{u}_n,t_n)\\
k_2 &= N(e^{-L\Delta t}\left(\hat{u}_n + \frac{\Delta t}{2}k_1\right),t_n + \frac{\Delta t}{2})\\
k_3 &= N(e^{-L\Delta t}\left(\hat{u}_n + \frac{\Delta t}{2}k_2\right),t_n + \frac{\Delta t}{2})\\
k_4 &= N(e^{-L\Delta t} \hat{u}_n + e^{-L\frac{\Delta t}{2}}\Delta t k_3),t_n+ \Delta t)\\
\hat{u}_{n+1} &= e^{-L\Delta t}\hat{u}_{n} + \frac{\Delta t}{6}\left( e^{-L\Delta t}k_1 + 2e^{-L\frac{\Delta t}{2}}\left(k_2 + k_3\right) + k_4\right)
\end{align*}

\State Nudging forward time integration
\begin{align*}\hat{v}_{n+1} = M(\hat{v}_n) + e^{-L\Delta t}\mu \Delta t  P_M\left( \hat{u}_n - \hat{v}_n\right)\end{align*}
\end{algorithmic}
\end{algorithm}

The nudging term is implemented numerically in \cref{alg:IFRK4} as an explicit term. 
This implementation of the feedback-control term that we used induces a CFL constraint required for the numerical scheme to remain stable, $\mu \lesssim \frac{2}{\Delta t}$. 
This CFL constraint can be removed by using an implicit formulation of the nudging term, which we do not consider in this work. 
It is worth noting that the nudging term cannot be implemented numerically using multi-stage methods, such as RK4, due to the need to interpolate in time (see e.g. \cite{Olson_Titi_2008_TCFD}).
The simulations featured in this work primarily use a value of $\mu = 100$, which was found to be sufficient for this study. 
For an in-depth examination of the role the value of the nudging parameter plays in convergence, see \cite{Carlson_Farhat_Martinez_Victor_2024_ISYNC}.

To test the effectiveness of nudging on the KdV equation \cref{eq:KdV} we used an ``identical twin'' experimental design. 
This experimental design is standard in the validation of data assimilation methods, where we first initialize a reference solution. 
This solution will be treated as the true solution which we will attempt to recover using observations. 
The observations will be given by the lowest $M$ Fourier modes. The observations are considered to be perfect with no distortions (i.e., no errors in the measurements) or stochastic noise.

It is worth mentioning that typically in the data assimilation literature for dissipative evolutionary equations, the equations are evolved forward in time long enough such that the solution remains in an absorbing set in a relevant phase space.  That is, we run solutions long enough forward in time such that they are uniformly bounded with respect to time in $L^2$, $H^1$, or another relevant norm for all future times, see e.g., \cite{Robinson_2001,Temam_1997_IDDS} for a full definition of absorbing sets as well as their properties. In this study we do not evolve the equations forward in time to large times, as solutions are not dissipative and so there is not necessarily an absorbing set that attracts all bounded solutions after a long enough time. In fact, as the $L^2$ norm is a conserved quantity for \cref{eq:KdV}, such an absorbing set cannot exist in $L^2$. We opted to simply evolve forward the initial profile a given amount of time, typically $t = 1$, in order to ensure that the reference solution has developed interesting dynamics prior to the observational period outside the observation domain. While solutions can develop small length scales not captured by the observed modes, they typically have a finite number of active Fourier modes, with the total number of modes active at a given time dependent on the initial condition, the current time, the total momentum in the system, and the dispersion parameter, $\delta^2$. 
An active Fourier mode in this sense means that $\abs{\hat{u}_k} \gg \ep$ or $\abs{\hat{u}_k} > 10^j \ep$ for some reasonable $j$, where $k\in \mathbb{Z}$ is the wavemode index and $\ep$ indicates machine precision.

\subsubsection{Classical KdV equation}\label{sect:KdV}
In this section we will examine how the data assimilation nudging algorithm fails for \Cref{eq:KdV}. The extent to which it works appears to be determined by whether the active Fourier modes of the initial profile are contained in the annulus of the observed Fourier modes, as we will see later on in \cref{fig:large time success}. We will start by considering examples such as in \Cref{prop:4} to illustrate how nudging fails.

To find a case where nudging fails, let us consider a solution initialized on a single wavemode
\begin{equation}
    u^{in}(x) = c \cos(k_0\pi x), k_0 \in \mathbb{Z}\setminus\{0\},\label{eq:cosk_0}
\end{equation}
where $c \in \mathbb{R}$ is an arbitrary constant with $c\neq 0$.
Note that by \cref{prop:2} we know that as this solution is initially $\frac{2}{k_0}$ periodic, it will remain $\frac{2}{k_0}$ periodic for all time $t\in \mathbb{R}$.

Using initial data of the form \cref{eq:cosk_0} we can now construct solutions to \cref{eq:KdV} which the nudging algorithm \cref{eq:KdV nudged} explicitly fails to recover. Given a number of observed modes $M$, we can set $k_0 \in \mathbb{Z} $ such that $\abs{k_0} \geq M+1$. Again, we emphasize that solutions initialized in this way (see \cref{eq:cosk_0}) are initially $\frac{2}{k_0}$ periodic and remain $\frac{2}{k_0}$ periodic due to \cref{prop:2}.
Noting that $M< k_0$, it follows immediately from \cref{prop:4} that $P_M u(t) \equiv 0$ where $u$ is the solution of \cref{eq:KdV} generated from the initial data prescribed by \cref{eq:cosk_0}. This shows that a fixed number of observed Fourier modes are insufficient to distinguish the zero solution from solutions initialized with suitable high frequency oscillations, and so one should not expect data assimilation to succeed. Indeed, if we choose $v$ to be the solution of \cref{eq:KdV nudged} initialized by the identically zero solution, then it is easy to deduce that $P_M (u-v) \equiv 0$ holds for all time and so the solution to \cref{eq:KdV nudged} simply corresponds to a solution to \cref{eq:KdV}, albeit with a different initial condition than the reference solution.

One can see computationally in \cref{fig:zero separation} that nudging fails to recover the solution $u(x,t)$, of \cref{eq:KdV} generated from the initial data prescribed by \cref{eq:cosk_0}. 
In \cref{fig:zero separation} the solution, $v(x,t)$, of \cref{eq:KdV nudged} is initialized with the identically zero solution. 
Exactly as one expects, we find that $v \equiv 0$ holds for the entire simulation. 
Similarly, we see in \cref{fig:zero control} the roles of $u$ and $v$ are reversed. 
That is, we instead initialize the solution to $v$ with \cref{eq:cosk_0} and the reference solution, $u$, is initialized with the identically zero solution. 
We see in \cref{fig:zero control} similar behavior occurs, as the feedback-control term remains zero for the entire simulation and both equations evolve as solutions of \cref{eq:KdV} with different initial conditions. 
While this is expected, it is nevertheless an interesting result as it shows that the initialization of the assimilated solution to \cref{eq:KdV nudged} matters. 
This fact is fundamentally different to many analytical results regarding the application of the nudging algorithm to dissipative equations, where the effect of the initial condition fades given enough time, contrary to what occurs in this case.

The previous counterexample illustrates that data assimilation methods cannot uniquely determine the observed reference solution when the observations are identically zero, but one may hope that more physically realistic (i.e. not identically zero) observations will yield better results.
This is true to some extent, however the level of success is highly dependent on the initial profile of the reference solution and one can similarly construct solutions that exhibit high frequency oscillations that nudging cannot recover. 
If given an analytic solution $u$ in $L^2$ of \cref{eq:KdV}, we can indeed construct infinitely many solutions where nudging fails to synchronize with the reference solution given a finite number of observed Fourier modes.

We will now demonstrate a general method of constructing solutions that nudging fails to synchronize.
Let $u$ be a solution to \cref{eq:KdV} such that $\norm{u^{in}}_{L^2} = K$. Here $u^{in} = u(0,x)$ is the initial condition of $u$. Note that the momentum is a conserved quantity of \Cref{eq:KdV}, and so $\norm{u}_{L^2} = K$ for all time. Now, let $M$ be given such that 
\begin{equation}\label{eq:M choice}
2\sum_{k \in \mathbb{Z}, \abs{k}\leq M} \abs{\hat{u}_k}^2 \geq K^2 - \ep.
\end{equation}
By Parseval's identity  $\norm{u}_{L^2}^2 = 2\sum_{k\in \mathbb{Z}} \abs{\hat{u}_k}^2 = K^2$. Here $\ep$ is a specified constant, for computations we will use $\ep \sim 10^{-16}$ to denote machine precision. The purpose of finding such an $M$ is to separate the wavemodes into active low modes and inactive high modes. At any fixed time $t>0$ such an $M$ is guaranteed to exist as $u$ is analytic and therefore the Fourier coefficients must decay exponentially fast in wavenumber.
In our computations we run only for a finite amount of time, so we take $M$ for $u$ to be the supremum over all such $M$'s taken across all simulated times. Without loss of generality we will assume that observation window consists of all $M$ modes, and so essentially the entire solution $u$ is known down to machine precision. We note here that $M$ is vital to the construction we outline below, thus the examples we construct work only for finite times. This is due to the fact that we require a uniform bound on the number of active wavemodes to extend this arbitrarily far in time.

Now, we can take the given solution $u$ and use it to generate artificial solutions for which nudging appears to fail. 
Consider the solution generated by the following initial data:
\begin{equation}\label{eq:general failure}
    \tilde{u}^{in} = c_1 u^{in} + c_2 \cos((2M+k)\pi x)
\end{equation}
Here $c_1$ and $c_2$ are positive constants chosen to enforce that $\norm{\tilde{u}^{in}}_{L^2} = K$ and $c_2 > \ep$. That is, we are modifying the initial profile to take some of the momentum out of the active modes and shifting it into to wavemode $2M+k$, where $k$ is an arbitrary non-negative integer. As $k$ is arbitrary, we can construct infinitely many such solutions, all of which have the same momentum as the given solution, $u$, but with a new active frequency in the initial profile. Note that the choice of using ${u}^{in}$ is arbitrary, and one can do this exact construction using the solution of $u$ evaluated at any given time.

We performed computational tests using the reference solution, $\tilde{u}$, of \cref{eq:KdV} with initial conditions prescribed by \cref{eq:general failure}. Here we took a solution, $u$, of \cref{eq:KdV} which was initialized simply with $u^{in} = \cos(\pi x)$. We found $M = 50$ was suitable for the construction in \cref{eq:general failure}. Once $M = 50$ was determined, we simply used $\tilde{u}^{in} =  \cos(\pi \cdot x) + \cos(100 \cdot \pi \cdot x)$, which was renormalized to ensure that 
$\norm{\tilde{u}^{in}}_{L^2} = \norm{u^{in}}_{L^2}$.
We see in \cref{fig:solenoid control} the result from applying nudging to attempt to recover the constructed reference solution $\tilde{u}$. Here the nudging solution, $v$, of \cref{eq:KdV nudged} is initialized to be identically zero.
We see in \cref{fig:solenoid control} that the low modes of $v$ synchronize with $\tilde{u}$ to a fixed level of precision, but the reference solution contains high frequency oscillations that $v$ never develops. The error in both the high and low modes appears approximately constant after a very brief initial decay in the observed modes. In additional tests (not depicted here) the same behavior in the $L^2$ error was observed up to time $t = 1000$ without any significant deviations in the error. One can clearly see activity in the frequencies immediately surrounding wavemodes $k = 100$ and $k = 200$ that is simply not present in the nudged solution. 
Not only is this high frequency activity never generated, but nudging features no mechanism to generate targeted high mode oscillations. That is, if one knew \emph{a priori} that the solution featured oscillations at wavemode $k = 100$, but one still only had observations of the first $50$ wavemodes, nudging does not have any mechanism to encourage oscillations at any unobserved frequencies.

It is worth mentioning that the construction of the initial profile in \cref{eq:general failure} is more complex than necessary. If one were to observe the first $M$ modes of a solution with a spike at wavemode $M+k$,  $k$ being a positive integer, then nudging would likely fail. Here, when we say ``spike'', we mean that looking at the $L^2$ spectrum one sees a local maximum at the specified wavemode. The construction of the initial profile in \cref{eq:general failure} is done to produce examples where one can see by looking at the spectrum that nudging fails to develop any of the high oscillations featured in the reference solution. Indeed, we see in \cref{fig:large time failure} that we see similar behavior in the $L^2$ error to that seen in \cref{fig:solenoid control}. In \cref{fig:large time failure} we initialize the reference solution, $u$, simply with $u^{in} = \cos(\pi \cdot x) + 0.001\cos(12\cdot \pi x)$, observing the first $10$ Fourier modes of $u$. The nudging solution, $v$, of \cref{eq:KdV nudged}, was initialized to be identically zero. We see in \cref{fig:large time failure} that the nudged solution still fails to develop all the high frequencies present in the reference solution, even at time $t = 1000$.

The simulations presented here indicate that nudging is ill-suited to capture extremely high oscillatory solutions to \cref{eq:KdV}, but it does not necessarily prohibit other methods of data assimilation from recovering the reference solution. Unlike in \cref{fig:zero control}, where the observed data is identically zero, it is not impossible in theory for data assimilation to tell the difference between two solutions initialized of the form \cref{eq:cosk_0}. That is, unlike the case of \cref{fig:zero control} if one considers two solutions $u_1$ and $u_2$ of \cref{eq:KdV} initialized as in \cref{eq:cosk_0} with separate $k_0$ values for $u_1$ and $u_2$, we have that $\norm{P_M(u_1 - u_2)}_{L^2}\neq 0$ for $t>0$, and to the best of our knowledge one cannot readily construct two such solutions of this form that always agree on every observed wavemode for all time.
This is due to the fact that dispersive term essentially causes the coefficients of the wavemodes to rotate in the complex plane around the origin with rotation speed $\delta^2k^3$. As each wavemode is rotating at a different speed one cannot rescale and shift momentum across wavemodes in a straightforward way such that the momentum backscattering due to the nonlinearity on the observed modes still agree at each timestep to produce a counterexample similar to the one present in \cref{fig:zero control,fig:zero separation}.

All of the cases where nudging fails are specific to the undamped and unforced, classical KdV \cref{eq:KdV}, which is not a dissipative dynamical system. The addition of damping will result in the energy being drained from these high oscillations, while the forcing injects energy into specified low wavemodes, resulting in nontrivial long time dynamics that eliminate these sorts of extremely high frequency oscillations, see \cref{sect:KdV damped} below.

\begin{figure}
    \centering
    \begin{subfigure}[t]{0.49\textwidth}
        \centering
        \includegraphics[width=\linewidth]{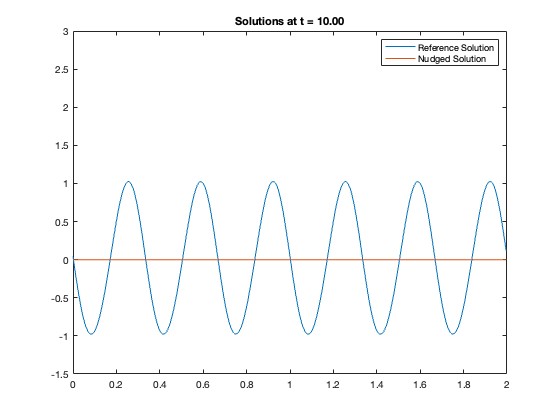}
        \label{fig:zero separation:a}
        \caption{Snapshot of solutions at time $t = 10$.}
        
    \end{subfigure}
    \hfill
    \begin{subfigure}[t]{0.49\textwidth}
        \centering
        \includegraphics[width=\linewidth]{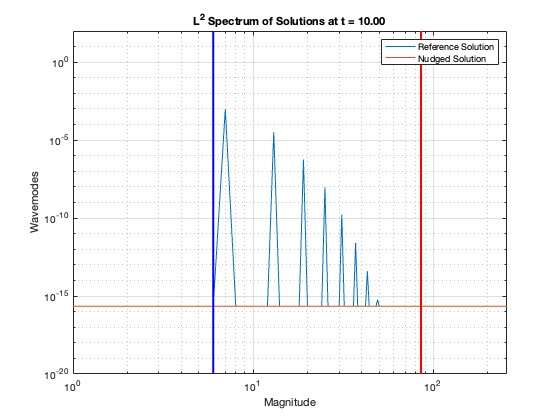}
        \label{fig:zero separation:b}
        \caption{Snapshot of $L^2$ spectrum of solutions at time $t = 10$.}
    \end{subfigure}

    \begin{subfigure}[t]{0.49\textwidth}
        \centering
        \includegraphics[width=\linewidth]{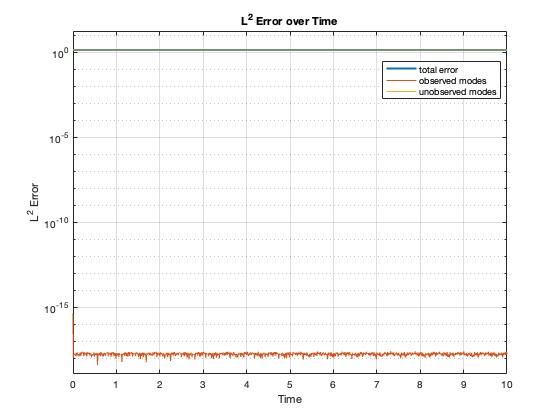}
        \caption{$\norm{u - v}_{L^2}$ over time split into observed low modes, unobserved high modes, and the total error.}
        \label{fig:zero separation:c}

    \end{subfigure}
    \hfill
    \begin{subfigure}[t]{0.49\textwidth}
        \centering
        \includegraphics[width=\linewidth,height=.24\textheight]{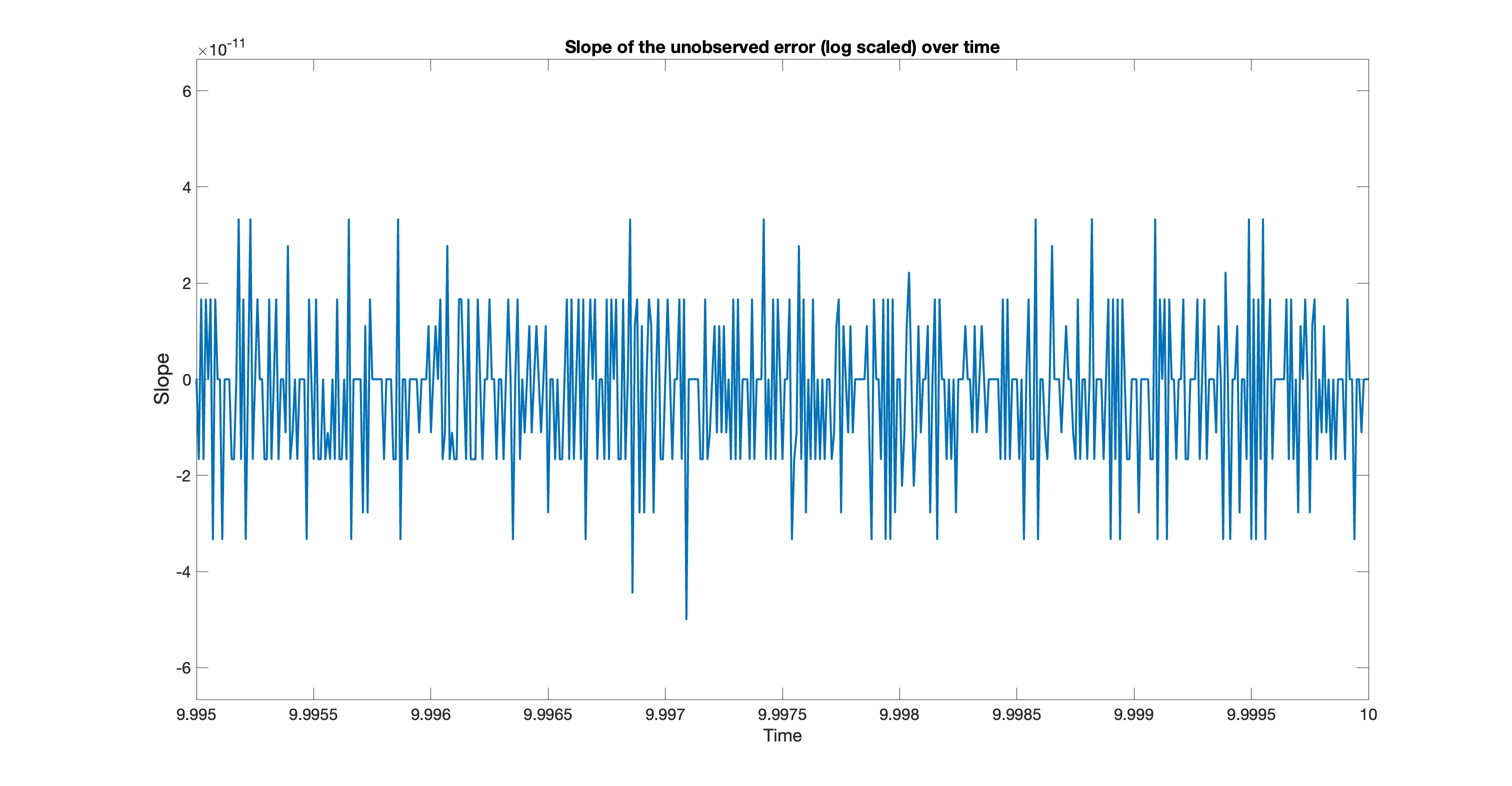}
        \caption{Approximate exponential decay rate of the high mode error for the last $500$ timesteps.}
        \label{fig:zero separation:d}
        
    \end{subfigure}
    
    \caption{Nudging with $5$ observed modes initialized as in \cref{eq:cosk_0} with $k_0 = 6$ fails to develop oscillations in the unobserved modes. Here $\mu = 100$ and $\delta = 0.075$.}
    \label{fig:zero separation}
\end{figure}

\begin{figure}
    \centering
    \begin{subfigure}[t]{0.48\textwidth}
        \centering
        \includegraphics[width=\linewidth]{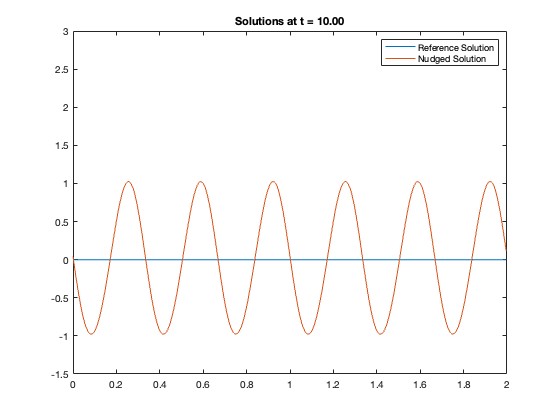}
        \caption{Snapshot of solutions at time $t= 10$.}
        \label{fig:zero control:a}

    \end{subfigure}
    \hfill
    \begin{subfigure}[t]{0.48\textwidth}
        \centering
        \includegraphics[width=\linewidth]{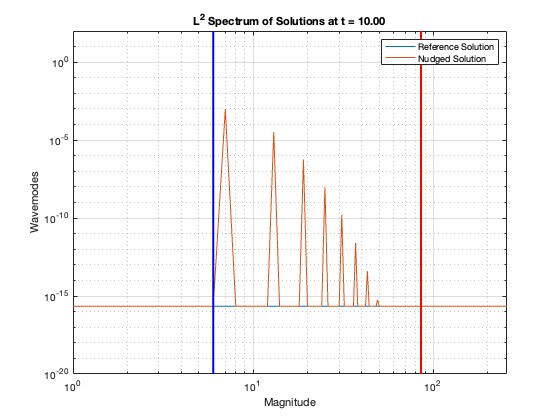}
        \caption{Snapshot of $L^2$ spectrum of solutions at time $t = 10$.}
        \label{fig:zero control:b}

    \end{subfigure}

    \begin{subfigure}[t]{0.48\textwidth}
        \centering
        \includegraphics[width=\linewidth]{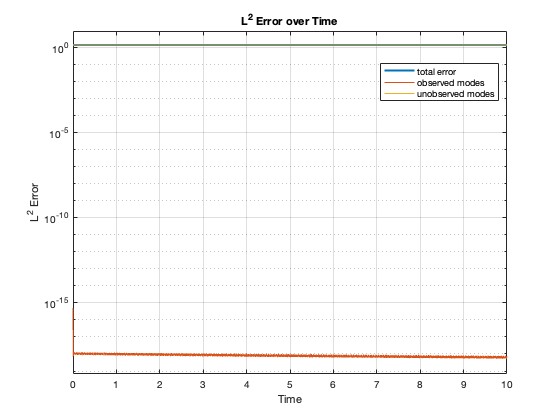}
        \caption{$\norm{u - v}_{L^2}$ over time split into observed low modes, unobserved high modes, and the total error. }        
        \label{fig:zero control:c}

    \end{subfigure}
    \hfill
    \begin{subfigure}[t]{0.48\textwidth}
        \centering
        \includegraphics[width=\linewidth,height=.24\textheight]{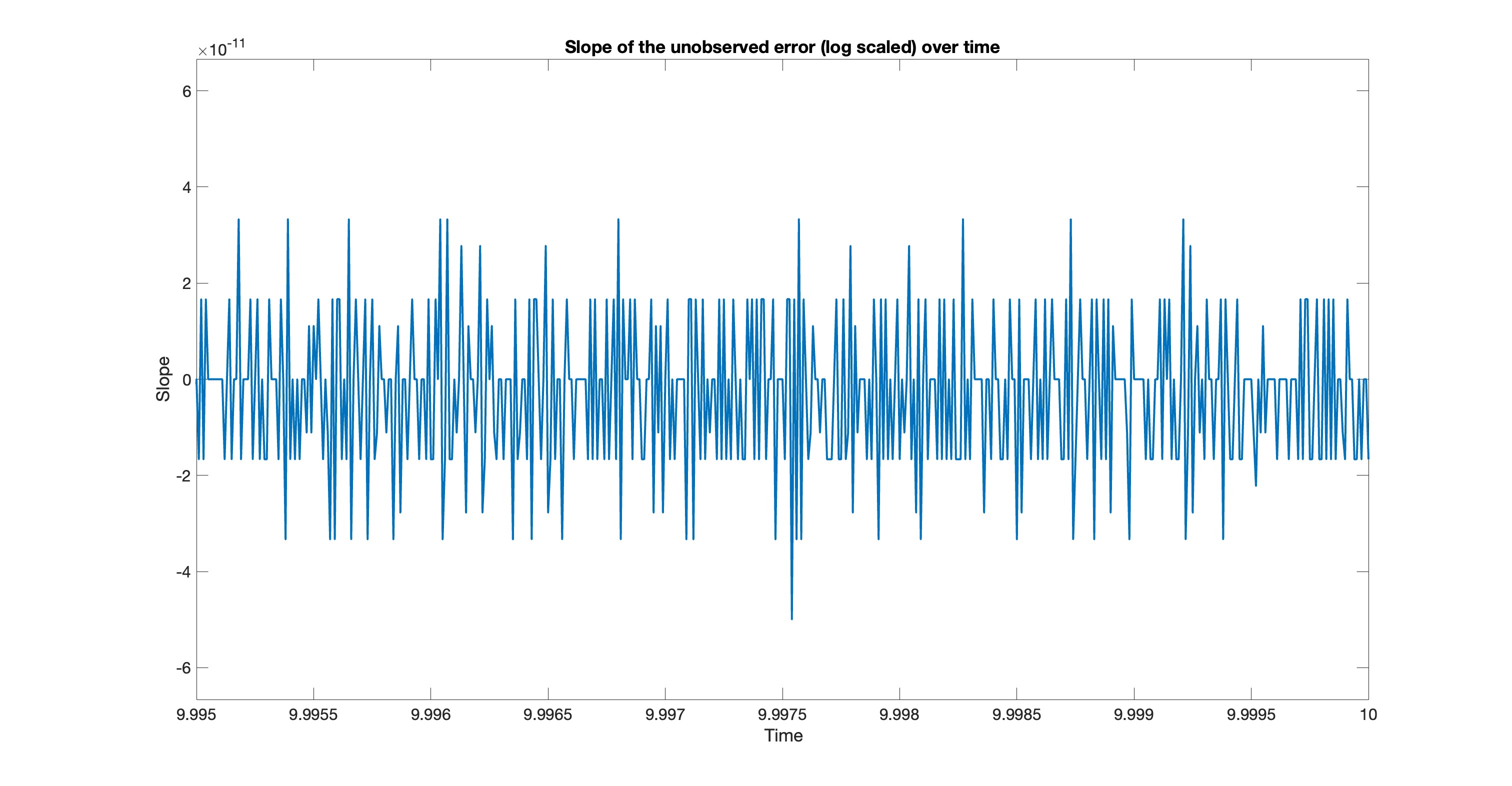}
        \caption{Approximate exponential decay rate of the high mode error for the last $500$ timesteps.}
        \label{fig:zero control:d}
        
    \end{subfigure}
    
    \caption{Nudging with $5$ observed modes initialized as in \cref{eq:cosk_0} with $k_0 = 6$ fails to control high oscillations towards the zero solution. Here $\mu = 100$ and $\delta = 0.075$.}
    \label{fig:zero control}
\end{figure}

\begin{figure}
    \centering
    \begin{subfigure}[t]{0.48\textwidth}
        \centering
        \includegraphics[width=\linewidth]{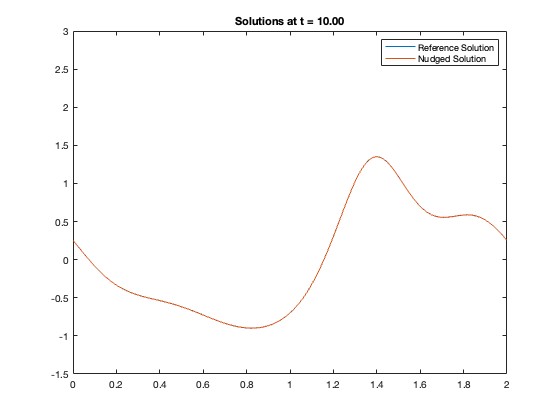}
        \caption{Snapshot of solutions at time $t = 10$.}
            \label{fig:solenoid control:a}

    \end{subfigure}
    \hfill
    \begin{subfigure}[t]{0.48\textwidth}
        \centering
        \includegraphics[width=\linewidth]{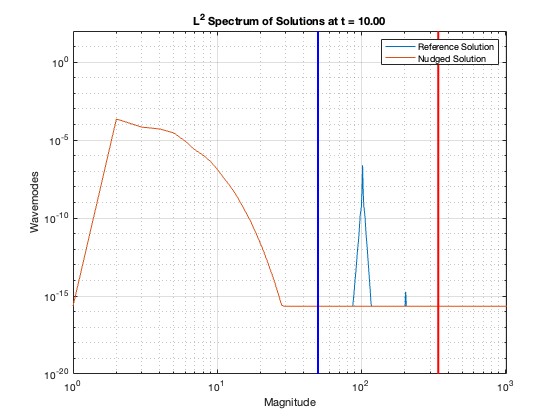}
        \caption{Snapshot of $L^2$ spectrum of solutions at time $t = 10$.}
            \label{fig:solenoid control:b}

    \end{subfigure}

    \begin{subfigure}[t]{0.48\textwidth}
        \centering
        \includegraphics[width=\linewidth]{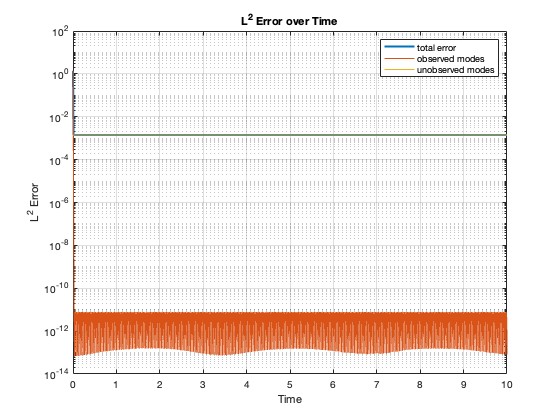}
        \caption{$\norm{u - v}_{L^2}$ over time split into observed low modes, unobserved high modes, and the total error.}
            \label{fig:solenoid control:c}

    \end{subfigure}
    \hfill
    \begin{subfigure}[t]{0.48\textwidth}
        \centering
        \includegraphics[width=\linewidth,height=.24\textheight]{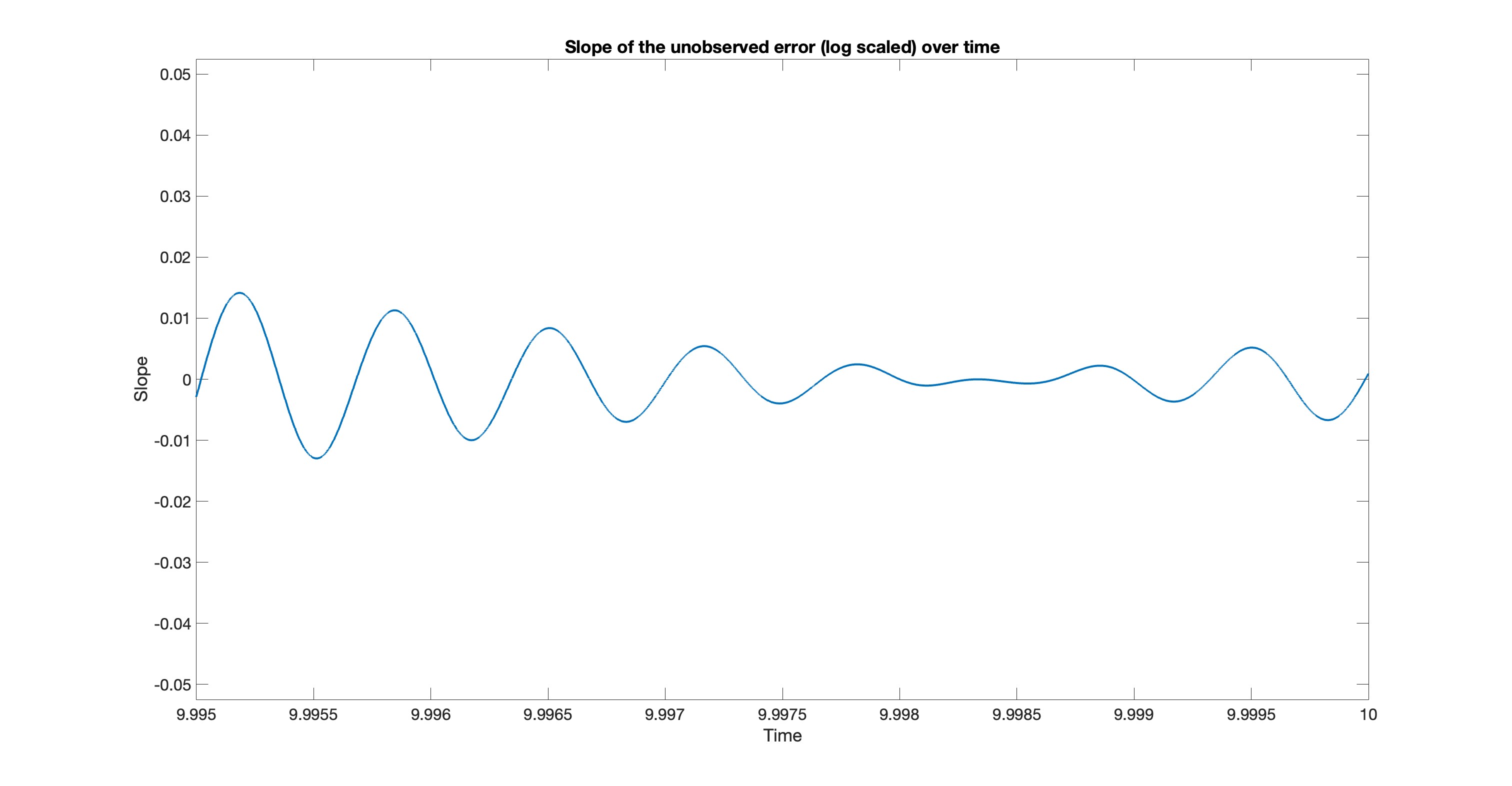}
        \caption{Approximate exponential decay rate of the high mode error for the last $500$ timesteps.}
            \label{fig:solenoid control:d}

    \end{subfigure}
    
    \caption{Nudging with $50$ observed modes and additional activity at wavemode $k_0 = 100$, initialized as in \cref{eq:general failure}. Here $\mu = 100$ and $\delta = 0.075$.}
    \label{fig:solenoid control}
\end{figure}

\begin{figure}
    \centering
    \begin{subfigure}[t]{0.48\textwidth}
        \centering
        \includegraphics[width=\linewidth]{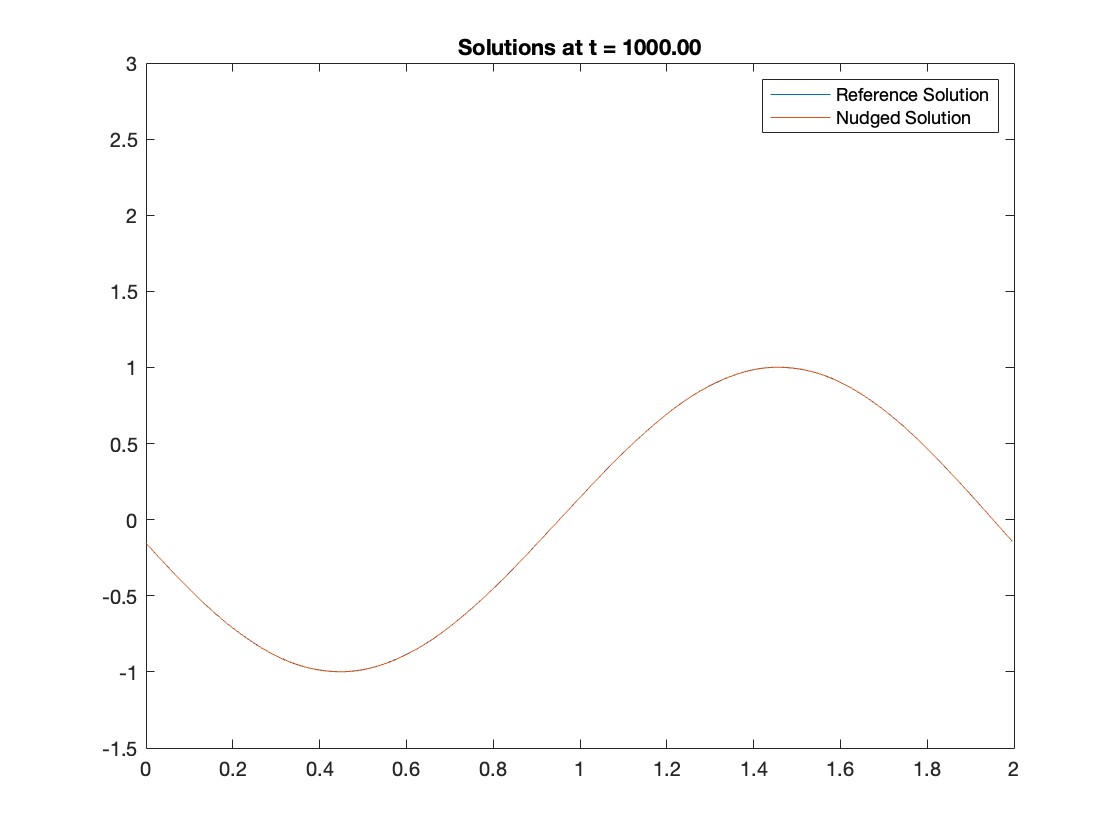}
        \caption{Snapshot of solutions at time $t=1000$.}
            \label{fig:large time failure:a}

    \end{subfigure}
    \hfill
    \begin{subfigure}[t]{0.48\textwidth}
        \centering
        \includegraphics[width=\linewidth]{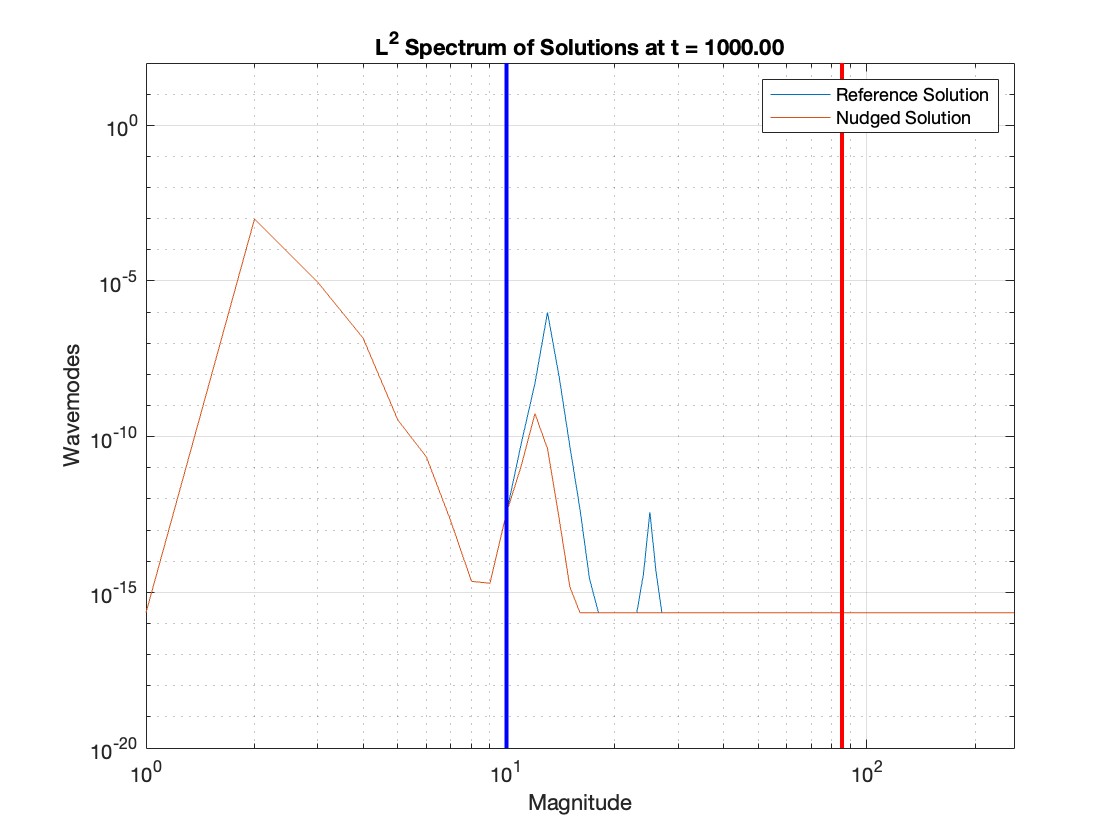}
        \caption{Snapshot of $L^2$ spectrum of solutions at time $t = 1000$.}
            \label{fig:large time failure:b}

    \end{subfigure}

    \begin{subfigure}[t]{0.48\textwidth}
        \centering
        \includegraphics[width=\linewidth,height=.24\textheight]{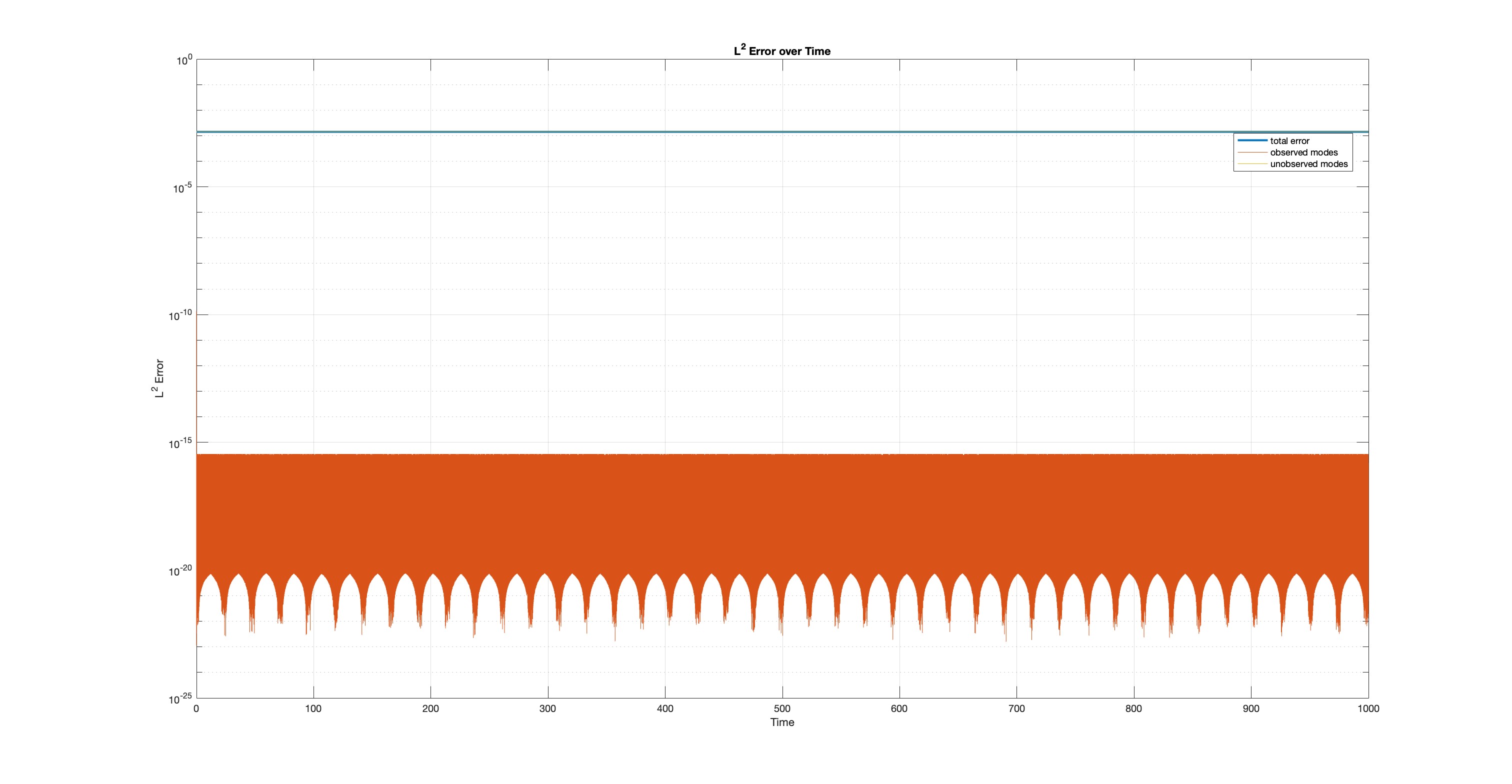}
        \caption{$\norm{u - v}_{L^2}$ over time split into observed low modes, unobserved high modes, and the total error.}
            \label{fig:large time failure:c}

    \end{subfigure}
    \hfill
    \begin{subfigure}[t]{0.48\textwidth}
        \centering
        \includegraphics[width=\linewidth]{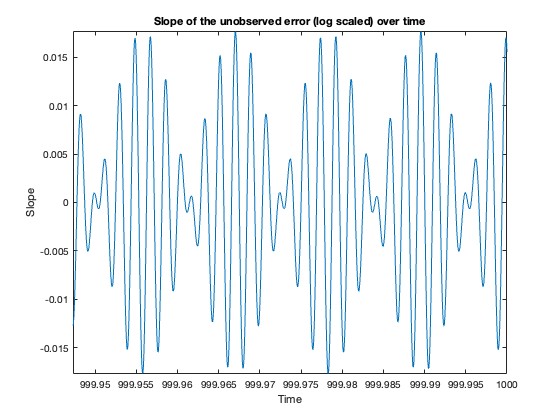}
        \caption{Approximate exponential decay rate of the high mode error for the last $500$ timesteps.}
            \label{fig:large time failure:d}

    \end{subfigure}
    
    \caption{Nudging with $10$ observed modes with $k_0 = 12$ fails to develop oscillations in the unobserved modes. Here $\mu = 100$, $\delta = 1$, and the reference solution is initialized as in \cref{eq:general failure} with the $M$ chosen to be smaller than in \cref{eq:M choice}.
    }
    \label{fig:large time failure}
\end{figure}

\begin{figure}
    \centering
    \begin{subfigure}[t]{0.48\textwidth}
        \centering
        \includegraphics[width=\linewidth]{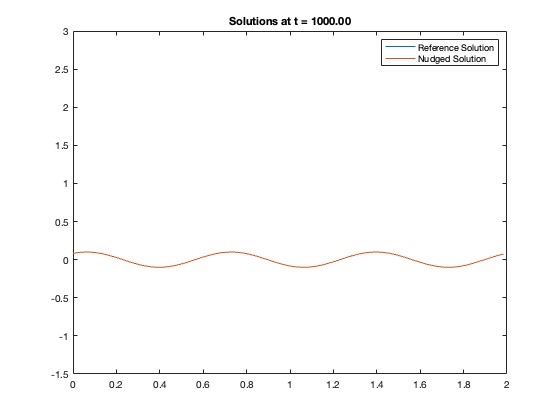}
                \caption{Snapshot of solutions at time $t = 1000$.}
        \label{fig:large time success:a}
        
    \end{subfigure}
    \hfill
    \begin{subfigure}[t]{0.48\textwidth}
        \centering
        \includegraphics[width=\linewidth]{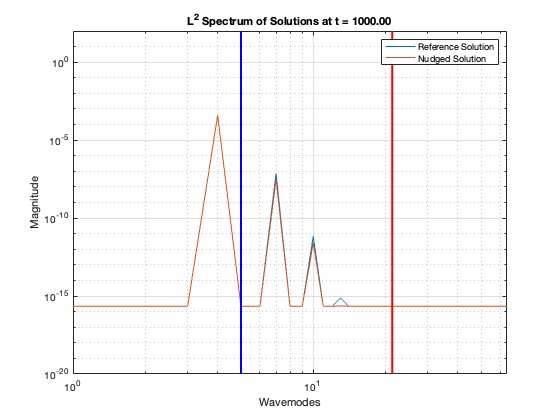}
                \caption{Snapshot of $L^2$ spectrum of solutions at time $t = 1000$.}
        \label{fig:large time success:b}
    \end{subfigure}

    \begin{subfigure}[t]{0.48\textwidth}
        \centering
        \includegraphics[width=\linewidth]{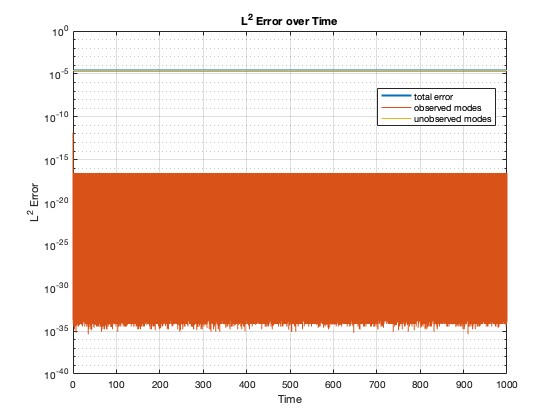}
                \caption{$\norm{u - v}_{L^2}$ over time split into observed low modes, unobserved high modes, and the total error.}
        \label{fig:large time success:c}
        
    \end{subfigure}
    \hfill
    \begin{subfigure}[t]{0.48\textwidth}
        \centering
        \includegraphics[width=\linewidth]{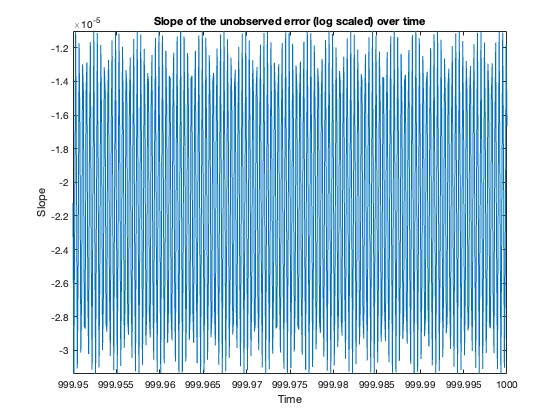}
        \caption{Approximate exponential decay rate of the high mode error for the last $500$ timesteps.}
        \label{fig:large time success:d}
        
    \end{subfigure}
    \caption{Nudging with $5$ observed modes with $k_0 = 3$ appears to work with an extremely slow rate of convergence. Here $\mu = 100$ and $\delta = 1$.}
        \label{fig:large time success}

\end{figure}

\subsubsection{Damped and Driven KdV}\label{sect:KdV damped}
In this subsection we examine the case of the damped and driven KdV equation, that is, we study the following equation:
\begin{equation}\label{eq:KdV damped and driven}
u_t + uu_x + \delta^2 u_{xxx} + \gamma u = f,
\end{equation}
subject to periodic boundary conditions.
Here $\gamma>0$ is the damping coefficient constant and $f$ is a body force that is given to be constant in time with zero spatial average.  The corresponding nudged damped and driven KdV equation is given by
\begin{equation}\label{eq:KdV damped and driven: nudged}
    v_t + vv_x + \delta^2 v_{xxx} + \gamma v = f + \mu P_M(u-v).
\end{equation}
In \cite{Jolly_Sadigov_Titi_2017} it was shown analytically, that given enough observed Fourier modes and large enough $\mu$, that the nudged solution converges to the reference solution in $L^2$ exponentially fast in time with exponent $\frac{\gamma}{4}$ for arbitrary choice of initial data $v^{in} \in H^2$ with spatial mean zero for \cref{eq:KdV damped and driven: nudged}.

In \cref{fig:damped and driven} we can see nudging effectively recovers a reference solutions initialized as in \cref{eq:cosk_0} exponentially fast in time. Here we utilize $\mu = 100$ and $\gamma = 0.1$ and the exponential rate of convergence is given by small perturbations from $\gamma$. Here we initialize the reference solution as in \cref{eq:cosk_0} with $k_0 = 12$, that is, we selected the initial profile $u^{in} = \cos(12\cdot\pi\cdot x)$. For our driving force, we used $f = \cos(8\cdot\pi\cdot x)$.  
In the simulations shown in this section we use $M = 10$, indicating that only the first 10 Fourier modes of the solution are observed. 
It is not surprising that convergence happens at the exponential rate proportional to the damping coefficient, $\gamma$, as this coincides with the results of \cite{Jolly_Sadigov_Titi_2017}.
We note that the exponential convergence is not entirely independent of $\mu$, the oscillatory component at initial times and on the observed modes is particularly sensitive to $\mu$ values, although this effect quickly fades as the error begins to decay at rate $\gamma$. 
Moreover, as one sees from the spectrum plots in \cref{fig:damped and driven} the primary barrier for convergence are the spikes in undetected frequencies, therefore it is not surprising that convergence to the reference solution happens at exactly the rate, $\gamma$, at which momentum is drained from these modes. 
Indeed, this behavior is typical of simulations of nudging methods, as the unobserved modes rely on the dissipation in the system in order to synchronize. The overall decay in error can typically be decomposed into two distinct phases, an initial rapid decay in error that is highly dependent on $\mu$, and a slower, but still exponentially fast in time, rate of convergence primarily driven by the dissipation in the system. For an in-depth study of the role of the particular value of $\mu$ see, e.g., \cite{Carlson_Farhat_Martinez_Victor_2024_ISYNC}.

The AOT algorithm successfully recovers the reference solution, $u$, to \cref{eq:KdV damped and driven} even when $u$ is initialized as in the counterexamples from the previous section \cref{eq:cosk_0,eq:general failure} given enough observed modes. 
Note that in \cref{fig:damped and driven} we see some behavior in the observed modes due to the presence of the body force. 
The body force can be disabled by setting it to be identically zero, however in that case data assimilation would not be required at all. 
This is because the damping term will drive \textit{any} initial condition to $0$, exponentially fast at the rate $\gamma$ with or without any form of data assimilation being applied. 
Additionally, the AOT algorithm is only guaranteed to recover the true solution of \cref{eq:KdV damped and driven} asymptotically in time, which in the case of zero body forcing is the zero solution. 
While the observed scales may be approximately recovered via the AOT algorithm, solutions initialized as in \cref{eq:cosk_0,eq:general failure} will exhibit large oscillations in the unobserved Fourier modes that only approximately vanish after advancing suitably far in time.
 
\begin{figure}
    \centering

    \begin{subfigure}[t]{0.3\textheight}
        \centering
        \includegraphics[width=\linewidth]{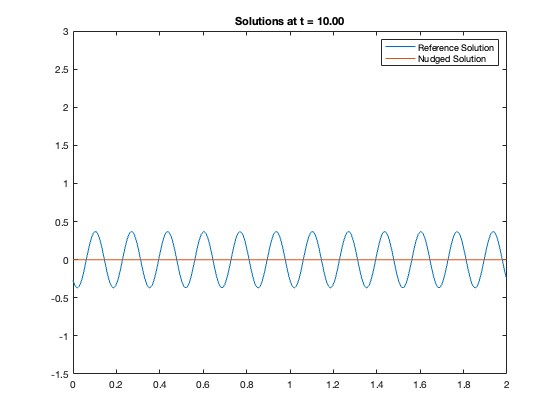}
        \caption{Snapshot of solutions at time $t = 10$.}
        \label{fig:damped and driven:a}
    \end{subfigure}
    \begin{subfigure}[t]{0.3\textheight}
        \centering
        \includegraphics[width=\linewidth]{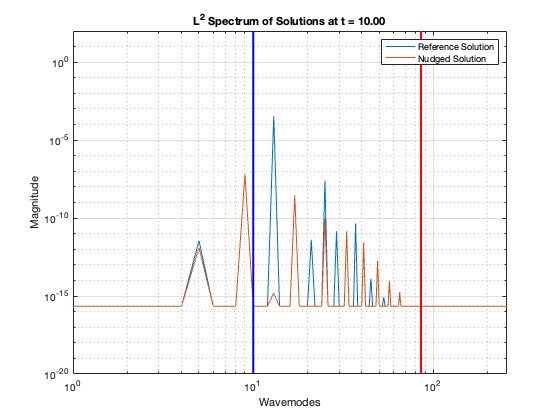}
        \caption{Snapshot of $L^2$ spectrum of solutions at time $t = 10$.}
        \label{fig:damped and driven:b}
    \end{subfigure}

    \begin{subfigure}[t]{0.3\textheight}
        \centering
        \includegraphics[width=\linewidth]{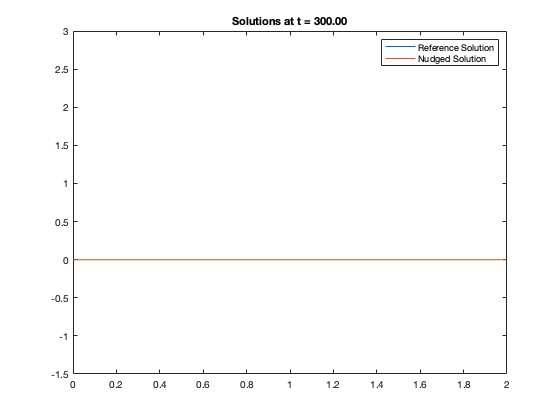}
        \caption{Snapshot of solutions at time $t = 300$.}
        \label{fig:damped and driven:c}
    \end{subfigure}
       \begin{subfigure}[t]{0.3\textheight}
        \centering
        \includegraphics[width=\linewidth]{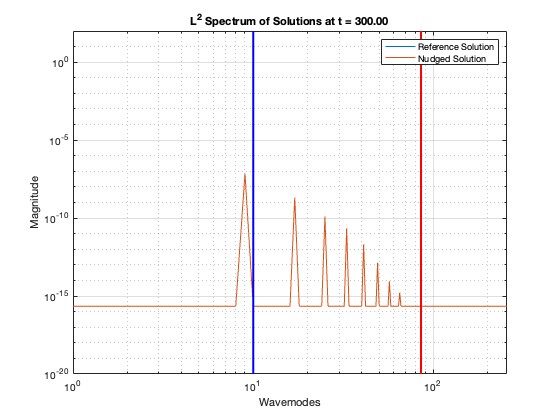}
        \caption{Snapshot of $L^2$ spectrum of solutions at time $t = 300$.}
        \label{fig:damped and driven:d}
    \end{subfigure}

    \begin{subfigure}[t]{0.3\textheight}
        \centering
        \includegraphics[width=\linewidth]{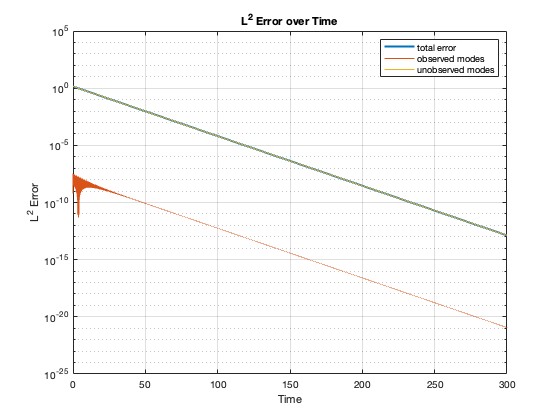}
        \caption{$\norm{u - v}_{L^2}$ over time split into observed low modes, unobserved high modes, and the total error.}
        \label{fig:damped and driven:e}
        
    \end{subfigure}
    \begin{subfigure}[t]{0.3\textheight}
        \centering
        \includegraphics[width=\linewidth]{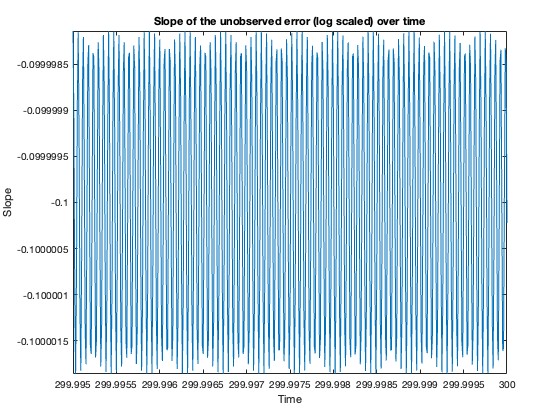}
        \caption{Approximate exponential decay rate of the high mode error for the last $500$ timesteps.}
        \label{fig:damped and driven:f}
    \end{subfigure}
     \caption{ Nudging with $10$ observed modes successfully recovers reference solution with nudging with $\mu = 100$ for $\gamma = 0.1$, $\delta = 1$, and $f = \cos(8\pi x)$. The nudged solution is initialized to be identically zero and the reference solution is initialized as in \cref{eq:cosk_0} with $k_0 = 12$. Note that the solution is not converging to zero, but features small amplitude oscillations. } 
    \label{fig:damped and driven}
\end{figure}

\section{2D PDEs - Euler} \label{sect:Euler}
\subsection{Analytical Results} \label{sect:Euler:preliminaries}
We now extend our results to the incompressible Euler equations on a 2D periodic domain. The equations are given as follows:
\begin{equation}
    \begin{split}
    \bu_t + \bu \cdot \nabla \bu &= -\nabla p, \\
    \div \bu &= 0. 
    \end{split}\label{eq: Euler}
\end{equation}
Here $\bu(x,t) = (u_1(\bx,t),u_2(\bx,t))$ is the velocity field and $p$ is the pressure.
The corresponding nudged equations are given by:
\begin{equation}\label{eq: Euler-nudged}\begin{split}
    \bv_t + \bv \cdot \nabla \bv &= -\nabla q + \mu P_M(\bu - \bv),
    \\
    \div \bv &= 0. \end{split}
\end{equation}
Here $\bv(x,t) = (v_1(\bx,t),v_2(\bx,t))$ is the velocity field, $q$ is the pressure, $\mu$ is the nudging parameter, and $M$ is the number of observed modes in the reference solution $u$ of \cref{eq: Euler}. 
As we are observing a 2D solution, $u$ of the system \cref{eq: Euler}, $P_M$ denotes the observations of all Fourier modes $\vect{k} \in \mathbb{Z}^2 \setminus \{(0,0)\}$ such that $\abs{\vect{k}}\leq M$.

We now describe the functional setting of the Euler equations, which will be useful in order to extend our results from KdV to this setting. For a more in-depth description of the functional formulation of the Euler equations see e.g. \cite{Temam_1997_IDDS, Constantin_Foias_1988, Foias_Manley_Rosa_Temam_2001}.
To begin, we denote the following space of test functions 
\begin{equation}
    \mathcal{V} := \left\{ \vect\varphi \in \mathcal{F}: \grad \cdot \vect\varphi = 0 \text{ and } \int_\Omega \vect\varphi(\bx) \,d\bx = 0 \right\}, 
\end{equation}
where $\mathcal{F}$ is the set of all two-dimensional vector-valued trigonometric polynomials defined on the domain 
$\Omega = \mathbb{T}^2 = [0,2\pi)^2$. 
We define $H = \overline{\mathcal{V}}^{L^2}$ and 
$V = \overline{\mathcal{V}}^{V}$, where $H$ and $V$ are standard Lebesgue and Sobolev spaces defined on $\Omega$. We define the inner products on $H$ and $V$ respectively by
\begin{equation}
 (\bu,\bv) = \sum_{m=1}^2 \int_\Omega u_mv_md{\bx} \quad\text{ and }\quad ((\bu,\bv)) = \sum_{m,n = 1}^2 \int_\Omega \pdv{u_m}{x_n}\pdv{v_m}{x_n}d{\bx},
 \end{equation}
which are associated with the norms $\norm{\bu}_{H} = \norm{\bu}_{L^2} = (\bu,\bu)^{1/2}$ and $\norm{\bv}_{V} = ((\bu,\bu))^{1/2}$. We note that the $V$ norm is the $H^1$ seminorm, which we use as the full norm for $V$ which is justified by the Poincar\'e inequality. We denote by $V^s$ the equivalent Sobolev space, given by $V^s = \overline{\mathcal{V}}^{V^s}$, where the $V^s$ norm is given by the $H^s$ seminorm.

We now apply the Leray-Helmholtz projector, $P_\sigma$, to \cref{eq: Euler,eq: Euler-nudged} to formally obtain the following evolution equations for $\bu$ and $\bv$, respectively
\begin{equation}\begin{split}
    \bu_t + B(\bu,\bu) &=0,\label{eq: Euler functional}\\ 
    \end{split}
\end{equation}
\begin{equation}\begin{split}
    \bv_t + B(\bv,\bv)  &=\mu P_\sigma P_M(\bu - \bv).\label{eq: Euler-nudged functional} 
    \end{split}
\end{equation}
Here we formally denote the nonlinear operator, $B(\vect{a},\vect{b}):= P_\sigma (\vect{a} \cdot \grad \vect{b})$ (cf. \cite{Temam_1997_IDDS, Constantin_Foias_1988}).

We recall the following theorem regarding global well-posedness of solutions to \cref{eq: Euler}
\begin{theorem}[Global well-posedness of \cref{eq: Euler} (see e.g. \cite{Majda_Bertozzi_2002})]
    Given $\bu^{in}\in V^{s}(\mathbb{T}^2)$ with $s\geq3$, then \cref{eq: Euler} is globally well-posed with solution 
    \[\bu \in C([0,T];V^s)\cap L^\infty([0,T]; H).\]
    Moreover $\norm{\bu(t)}_H = \norm{\bu^{in}}_H$ for all $t\geq 0$.
\end{theorem}

We now present the analogues of the results from \cref{sect:KdV} for the 2D incompressible Euler equations. 
As the proofs follow similarly to their counterparts in \cref{sect:KdV}, we summarize them into the following key statements without detailed proofs.

\begin{proposition}\label[proposition]{prop:1 - Euler}
Let $k\in \mathbb{Z}\setminus \{0\}$  and let $\bu^{in} \in V^3$  be given with zero spatial mean, satisfying for all $\bx\in \mathbb{R}^2$ and each $j = 1,2$,
\[ \bu^{in}(\bx+\frac{L}{k}\vect{e}_j) = \bu^{in}(\bx).\]
If $\bu(\bx,t)$ is the solution of \cref{eq: Euler} with initial data $\bu^{in}$, then for all $\bx \in \mathbb{R}^2$, $t\in \mathbb{R}$, and each $j = 1,2$, $\bu(\bx,t)$ satisfies
\[
\bu(\bx+\frac{L}{k}\vect{e}_j,t) = \bu(\bx,t).
\]
In particular, the space of functions periodic with respect to the lattice $k\mathbb{Z}^2 \setminus \{\bzero\}$ is invariant under the solution of \cref{eq: Euler}.
\end{proposition}

\begin{proposition}\label[proposition]{prop:2 - Euler}
    Let $k\in \mathbb{Z}\setminus \{0\}$ and let $\bu^{in}_j \in V^3$, for $j = 1,2$, satisfying the conditions of \cref{prop:1 - Euler}. 
    Moreover, assume $\norm{\bu^{in}_1}_{H} \neq \norm{\bu^{in}_2}_{H}$ and let $\bu_j(x,t)$ be the corresponding solutions of \cref{eq: Euler} with initial data $\bu_j^{in}$.
    Then for any $M<\abs{k}$ and all $t\geq 0$,
    \begin{equation}
        P_M(\bu_1(\cdot,t) - \bu_2(\cdot,t)) = 0.
    \end{equation}  
    and 
    \begin{equation}
        \limsup_{t\to\infty} \norm{\bu_1(\cdot,t) - \bu_2(\cdot,t)}_{H} > 0.
    \end{equation}
\end{proposition}

Now, it follows immediately from \cref{prop:2 - Euler} that solutions to \cref{eq: Euler} do not have the finitely many determining modes property:
\begin{theorem}\label{thm:determining modes - Euler}
The incompressible Euler equations \cref{eq: Euler} subject to periodic boundary conditions on $\mathbb{T}^2$ do not possess the finitely many determining modes property in $H$.    
\end{theorem}

\begin{corollary}\label[corollary]{cor: nudging fails Euler}
    For a fixed number of observed Fourier modes, $M$, there exists solutions $\bu(\bx,t)$ of \cref{eq: Euler} that cannot be recovered by \cref{eq: Euler-nudged} with arbitrary initial condition $\bv^{in}$.
\end{corollary}

In a future work, we will extend these results to the damped and driven Euler equations, which introduce additional complexities due to the external forcing and damping terms.


\subsection{Computational Results}\label{sect:Euler:computational results}

In this section we will detail the numerical methods we utilize to simulate \cref{eq: Euler,eq: Euler-nudged}.  For this computational study we implemented the incompressible Euler equations on the 2D periodic torus $\mathbb{T}^2 = [0,2\pi)\times[0,2\pi)$, subject to periodic B.C. with period $2\pi$ using pseudo-spectral methods. We implement the equations by using  the explicit Euler time-stepping scheme with the nonlinear computed explicitly. 
In the simulations in this chapter we utilize pseudo-spectral methods (see e.g.,  \cite{Canuto_Hussaini_Quarteroni_Zang_2006,Peyret_2013_spectral_book,Shen_Tang_Wang_2011} for textbook descriptions of these methods), that is, derivatives were computed using Matlab's implementation of the n-dimensional fast Fourier transform (\texttt{fftn}, which relies on FFTW).  We used the standard 2/3 dealiasing rule to compute the nonlinear term, that is, we zeroed out the modes above wavenumber $N/3$, where $N$ is number of gridpoints in one direction of the periodic box.  Moreover, to reduce the number of FFT operations, the Basdevant formula was used (see \cite{Emami_Bowman_2018} for further discussion). We utilize the spatial resolution $N = 2^8$ and time step $\Delta t = 0.001$.

We remark that while our analytical results illustrate the failure of nudging to recover the true solution, effective testing of the AOT algorithm in recovering solutions to \cref{eq: Euler} remains problematic. This is due to the fact that the AOT algorithm is typically shown to recover dissipative dynamical systems only asymptotically in time. Thus if one wishes to faithfully demonstrate the effectiveness of the AOT algorithm it requires running simulations of \cref{eq: Euler} for long amounts of time. This in itself is problematic, as it was shown in \cite{Elgindi_Hu_Sverak_2017} that solutions to \cref{eq: Euler} supported on a finite number of Fourier modes for all times $t\in [0,\infty)$ must necessarily be constant in time. 
This result directly shows that Fourier-based pseudo-spectral schemes cannot adequately approximate non-steady state solutions to \cref{eq: Euler} over arbitrarily long time intervals as they necessarily need to resolve the developing length scales finer than the fixed spatial resolution. 
This issue could potentially be alleviated by using methods with adaptive mesh refinement, however at some point the computer running the simulation will run out of the resources necessarily to adequately capture the developing fine length scales.
However, steady-state solutions are in fact enough to illustrate the failing of the AOT algorithm on the incompressible Euler equations.

For our reference solution we utilize the Taylor-Green vortex given with frequency $k = 15$ and $c = 10^{-4}$, given as follows (see also \cref{fig:Euler:Reference})
\begin{equation}\label{eq: Taylor-Green}
    \begin{cases}
        u_1(\vect{x}, t) = c\sin(k\vect{x}_1)\cos(k\vect{x}_2),\\
        u_2(\vect{x}, t) = - c\cos(k\vect{x}_1)\sin(k\vect{x}_2).
    \end{cases}
\end{equation}
We initialize the nudged equation with initial data identically zero $\vect{v}^{in} \equiv \vect{0}$, with $\mu = 100$. 
The results of our simulations can be seen below in \cref{fig:Euler:21,fig:Euler:22}.

We see in \cref{fig:Euler error 21} that observing all Fourier modes $\vect{k}$ with $\abs{\vect{k}} \leq 21$ is insufficient to recover the true solution. In the spectral plots \cref{fig:Euler spectrum 21,fig:Euler spectrum 22} the vertical lines indicate the observational frequency and dealiasing cutoff frequencies, given in blue and red, respectively. While $M = 21$ is insufficient to recover the true solution, we find that $M = 22$ is sufficient, which is of course not surprising as this is exactly the regime in which the entire solution is in fact observed. When $M$ is chosen smaller than the frequency used in the construction of the Taylor-Green vortex the solution $\bv(\bx,t) \equiv \vect{0}$ holds for all time $t\in [0,1]$.

\begin{figure}
    \centering
    \includegraphics[width=0.95\linewidth]{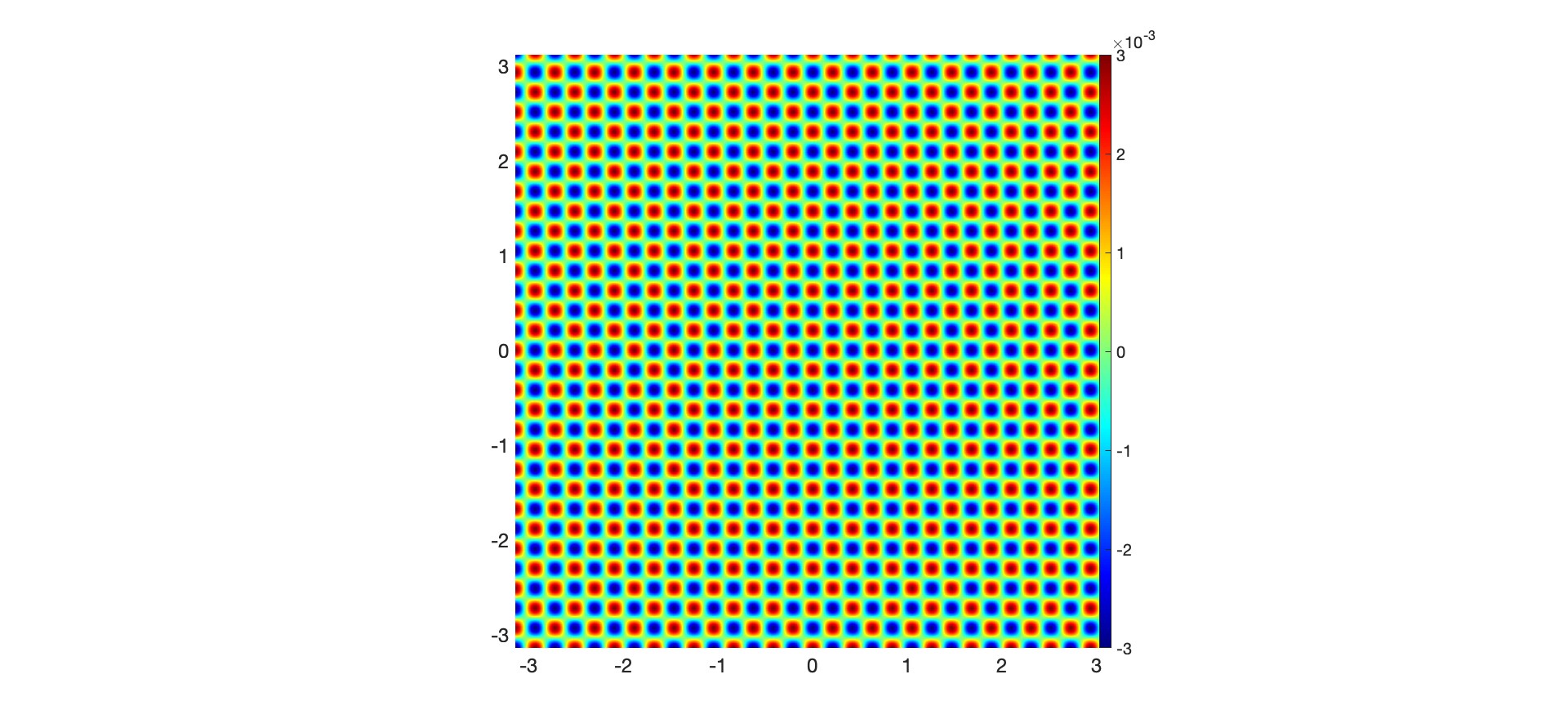}
    \caption{Reference solution given by Taylor-Green steady state with active frequency $k = 15$ and amplitude $c = 10^{-4}$. }
    \label{fig:Euler:Reference}
\end{figure}

\begin{figure}
    \centering

    \begin{subfigure}[t]{0.95\textwidth}
        \centering
        \includegraphics[width=\linewidth]{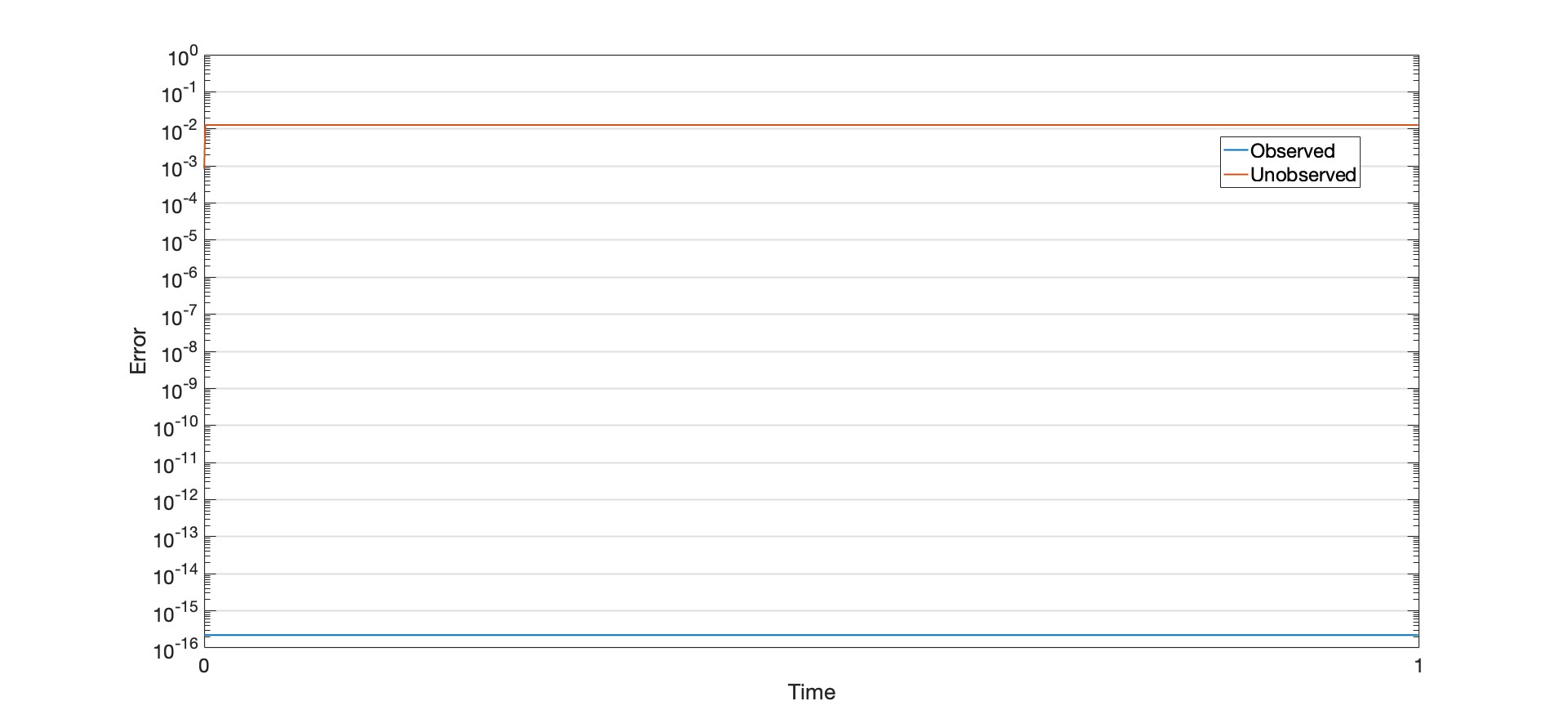}
        \caption{Error over time for $\mu = 100$ with $M = 21$ observed modes.}
        \label{fig:Euler error 21}
    \end{subfigure}
    \begin{subfigure}[t]{0.95\textwidth}
        \centering
        \includegraphics[width=\linewidth]{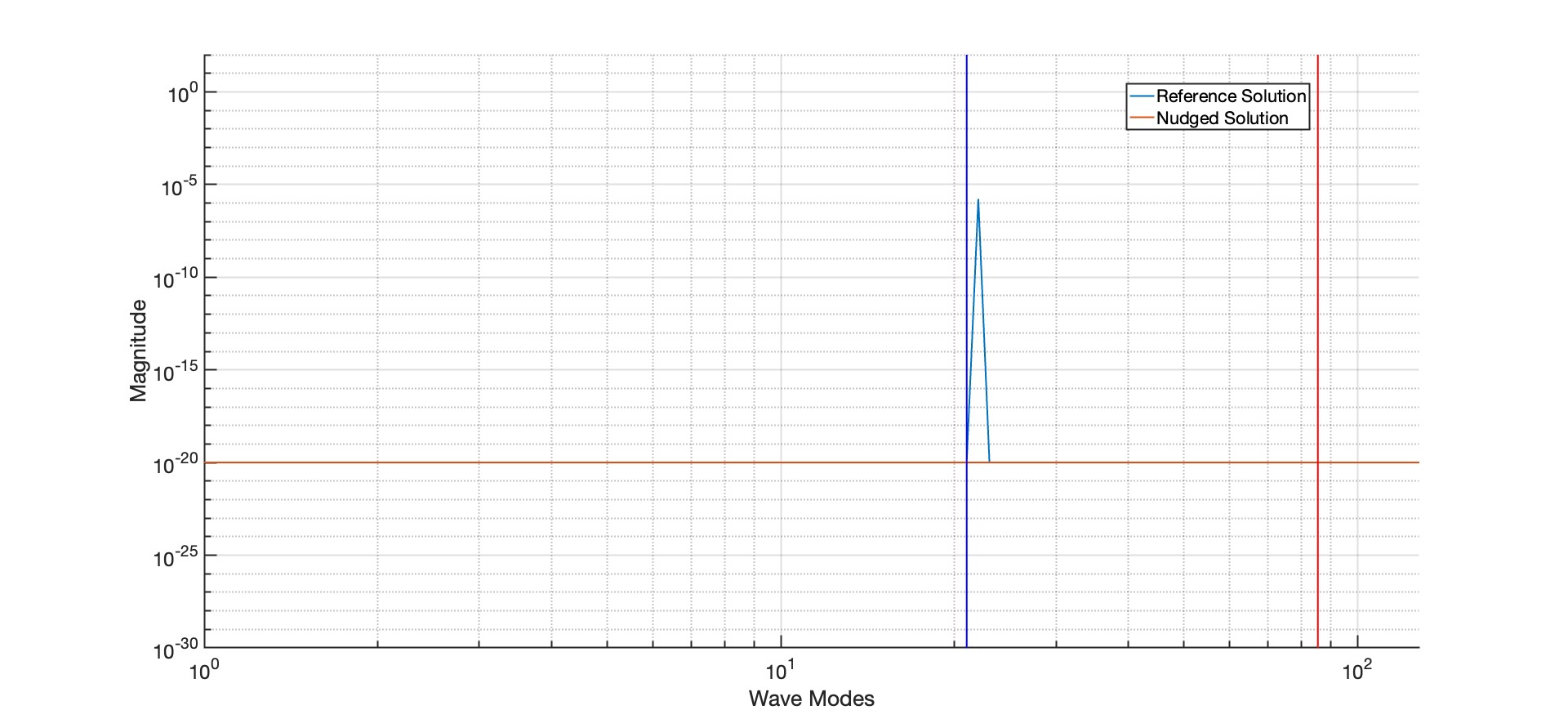}
        \caption{Snapshot of $L^2$ spectrum of solutions at time $t = 1$.}
        \label{fig:Euler spectrum 21}
    \end{subfigure}
 \caption{ Nudging with $21$ observed modes fails to recover reference solution with nudging with $\mu = 100$. The nudged solution is initialized to be identically zero and the reference solution is initialized as in \cref{eq: Taylor-Green} with $k = 15$. Note that $M = 21$ gives the radius of the observed frequencies, and the effective frequency for $k = 15$ is in fact $\abs{(15,15) } = \sqrt{2}(15) \approx 21.21.$} 
\label{fig:Euler:21}
\end{figure}

\begin{figure}
    \centering

    \begin{subfigure}[t]{0.95\textwidth}
        \centering
        \includegraphics[width=\linewidth]{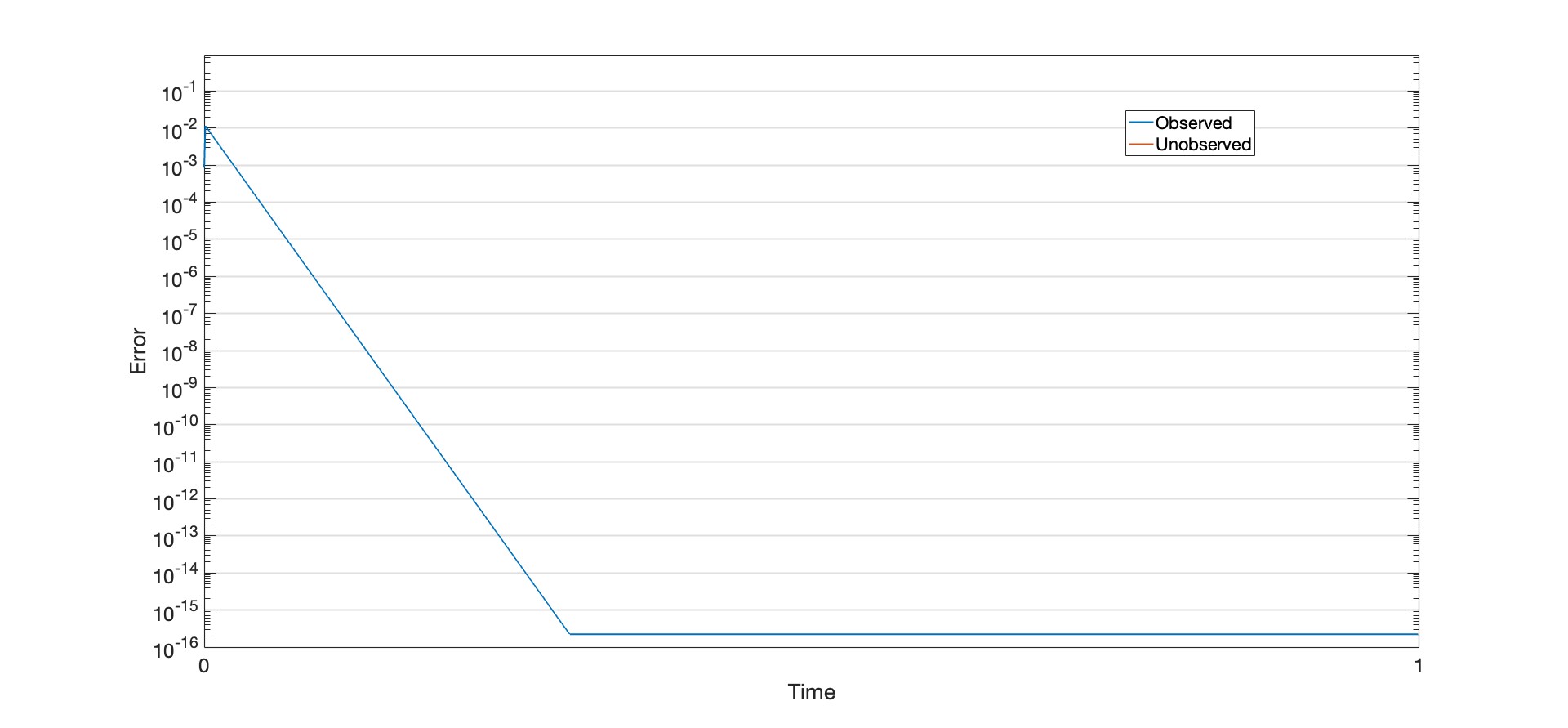}
        \caption{Error over time for $\mu = 100$ with $M = 22$ observed modes.}
        \label{fig:Euler error 22}
    \end{subfigure}
    \begin{subfigure}[t]{0.95\textwidth}
        \centering
        \includegraphics[width=\linewidth]{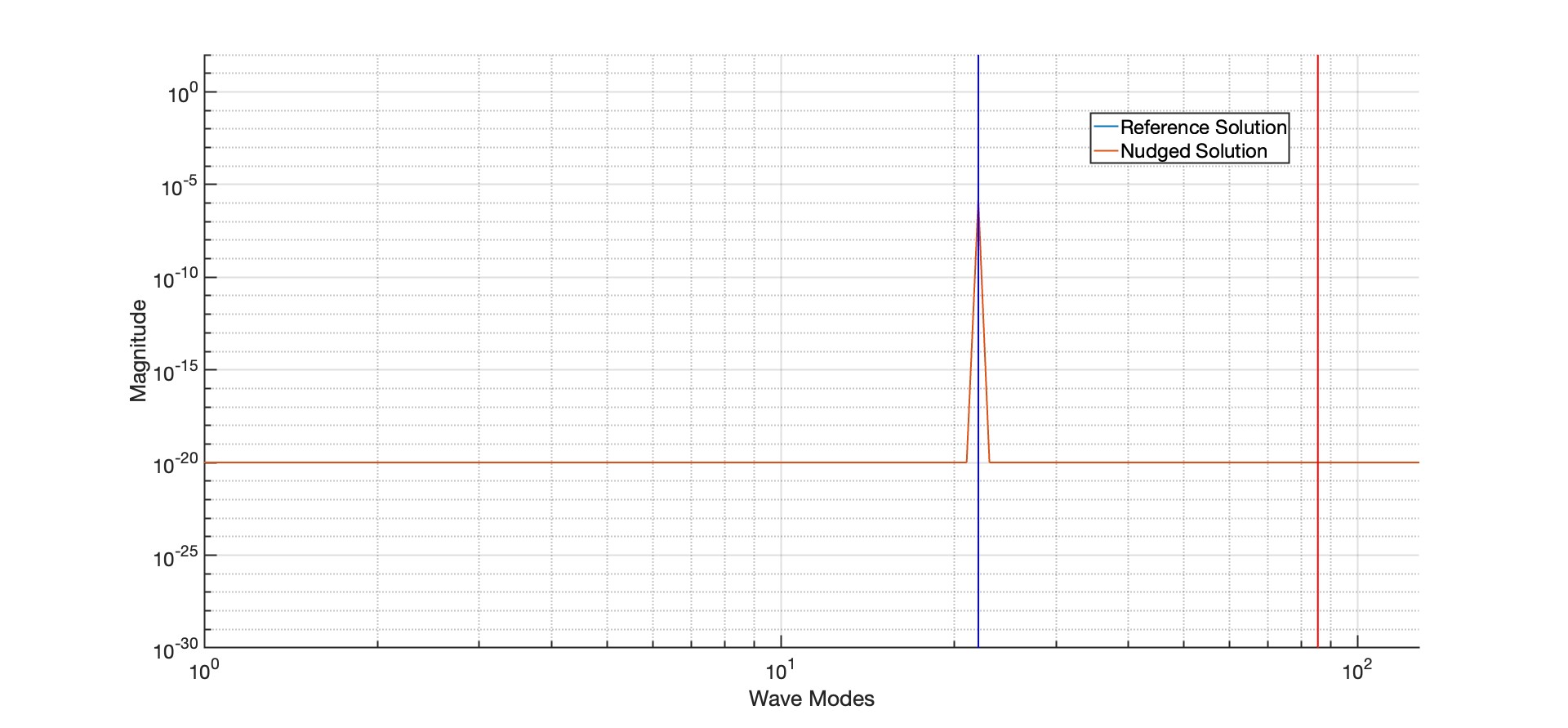}
        \caption{Snapshot of $L^2$ spectrum of solutions at time $t = 1$.}
        \label{fig:Euler spectrum 22}
    \end{subfigure}
 \caption{ Nudging with $22$ observed modes fails to recover reference solution with nudging with $\mu = 100$. The nudged solution is initialized to be identically zero and the reference solution is initialized as in \cref{eq: Taylor-Green} with $k = 15$. Note that $M = 22$ gives the radius of the observed frequencies, and the effective frequency for $k = 15$ is in fact $\abs{(15,15) } = \sqrt{2}(15) \approx 21.21.$} 
\label{fig:Euler:22}
\end{figure}

\section{Conclusion}\label{sect:conclusion}
Based on the results of the simulation presented in \cref{sect:KdV:computational results} and \cref{sect:Euler:computational results}, we conclude that the AOT data assimilation algorithm is unsuitable for recovering non-dissipative or partially dissipative dynamical systems, such as the incompressible Euler and KdV equations, \cref{eq: Euler,eq:KdV}, respectively, and the partially dissipative Lorenz system, \cref{eq:Lorenz}. 
In fact, we conclude that the pathological counterexamples used in \cref{sect:KdV,sect:Euler} present an insurmountable challenge for \emph{any} method of data assimilation to recover. 
In particular, the counterexample we utilize in \cref{fig:zero separation,fig:Euler:21} will fail for any method of data assimilation, as the true solution has no discernible impact on the observed modes, and so it is impossible to reliably tell the true solution apart from the zero solution (or one of an infinite number of solutions one can readily construct) based solely on the observations. 
Moreover, similar counter examples can be readily constructed for other non-dissipative dynamical systems, as the counter examples we considered in this work are not specific to the KdV or Euler equations, but rather are a consequence of the invariance of certain periodic subspaces under the flow. 
The particular solutions considered in \cref{prop:4,prop:2 - Euler} were constructed due to the periodic invariance induced by the convolutional Fourier mode structure of the nonlinear term (see \cref{cor,prop:1 - Euler}).

In addition to showing that data assimilation simply does not work for certain  pathological counterexamples, we note that the extent to which data assimilation does work for the classical KdV equation may be highly dependent on the choice of initial data for the nudged system, $v^{in}$. 
Without dissipation in the dynamical system, the errors in the initial data may persist indefinitely, preventing the accurate recovery of the reference solution.
Indeed, while it was proved in \cite{Azouani_Olson_Titi_2014} that for 2D NSE the choice of initial data $v^{in}$ was arbitrary, we have seen here that data assimilation will simply fail if not initialized properly, as it does in \cref{fig:zero separation}. 
Overcoming the problem of initialization is one of the core strengths of the nudging algorithm, and so it is worth noting that this property fails to hold for a system without finitely many determining modes.

It is worth mentioning that that data assimilation could be made to work if one were to assume more of the true solution. For instance, if one were to assume that the solution was supported on an annulus in Fourier space, one could eliminate all the examples that we study here. We caution against doing this, however, as imposing these sorts of restrictions on the solution can trivialize the dynamics of equations in a way that may not be physical.

\FloatBarrier

\section*{Acknowledgments}
\noindent

This research was supported in part by NPRP grant \# S-0207-359200290 from the Qatar National Research Fund (a member of Qatar Foundation). The work of E.S.T. was also supported by King Abdullah University of Science and Technology (KAUST) Office of Sponsored Research (OSR) under Award No. OSR-2020-CRG9-4336.

\bibliographystyle{plain}
\bibliography{VictorBiblio}
\end{document}